\documentclass[a4paper, 11pt]{article}

\usepackage{fullpage}

\usepackage{amsmath}
\usepackage{amsthm}
\usepackage{amsfonts}
\usepackage{tikz-cd}
\usepackage[utf8x]{inputenc}
\usepackage{enumerate}
\usepackage[english]{babel}
\usepackage{amssymb}
\usepackage{esint}
\usepackage{hyperref}
\usepackage{bm}
\usepackage{bbm}
\usepackage[mathscr]{eucal}

\theoremstyle{plain}
\newtheorem{theorem}{Theorem}[section]
\newtheorem{lemma}[theorem]{Lemma}

\newtheorem{proposition}[theorem]{Proposition}
\theoremstyle{definition}
\newtheorem{definition}[theorem]{Definition}
\theoremstyle{remark}

\numberwithin{equation}{section}

\DeclareMathOperator{\divv}{div}

\DeclareMathOperator*{\essinf}{ess\:inf}
\DeclareMathOperator{\trace}{Tr}

\def\softd{{\leavevmode\setbox1=\hbox{d}%
		\hbox to 1.05\wd1{d\kern-0.4ex{\char039}\hss}}}

\allowdisplaybreaks[1]

\title{ On well-posedness of quantum fluid systems \\ in the class of dissipative solutions}
\author{Danica Basari\'{c} \thanks{The work of D.B. was funded from the Czech Science Foundation (GA\v{C}R), Grant Agreement 21-02411S. The institute of Mathematics of the Czech Academy of Sciences is supported by RVO:67985840. $^ \dagger$ The research of T.T. was supported by NSFC:11801138} \and Tong Tang\ $^ \dagger$}
\date{}

\begin{document}
	
	\maketitle

	\begin{center}
		$^{*}$ Institute of Mathematics of the Czech Academy of Sciences \\
		\v{Z}itn\'{a} 25, 115 67 Praha 1, Czech Republic\\[0.1cm]
		basaric@math.cas.cz
	\end{center}
	
	\vspace{0.05cm}
	
	\begin{center}
		$^ \dagger$ School of Mathematical Science,\\
		Yangzhou University, Yangzhou 225002, P.R. China\\[0.1cm]
		tt0507010156@126.com
	\end{center}

	\vspace{0.2cm}

	\begin{abstract}
		The main objects of the present work are the quantum Navier--Stokes and quantum Euler systems; for the first one, in particular, we will consider constant viscosity coefficients. We deal with the concept of dissipative solutions, for which we will first prove the weak-strong uniqueness principle and afterwards, we will show the global existence for any finite energy initial data. Finally, we will prove that both systems admit a semiflow selection in the class of dissipative solutions.
	\end{abstract}

	
	\vspace{0.2 cm}
	

	\section{Introduction}
	
	At temperatures close to absolute zero, quantum effects appear relevant in the motion of some fluids: instead of individual atoms bouncing around, the particles move like one single body and, as a consequence of the vanishing viscosity, the fluid start to  ``creep" along the surfaces of its container, coming out of it if the latter is not properly sealed. This bizarre phenomena is just one of the many applications that motivate the study of \textit{quantum fluid dynamics}: it provides useful tools for understanding not only the behaviour of atomic Bose--Einstein condensates and the transition of the aforementioned fluids into zero-viscosity ones (\textit{superfluids}) \cite{LofMor}, but also the mechanics of quantum semiconductors \cite{FerZho} and the trajectories arising from the de Broglie--Bohm theory \cite{Wya}.
	
	\subsection{The system}
	
	Motivated by the Thomas--Fermi--Dirac–Weizs\"{a}cker density functional theory \cite{ZarTso}, the motion of a quantum fluid can be modelled starting from the classical systems describing viscous or inviscid fluids and adding an extra term containing the Bohm quantum potential \cite{SlaTse}
	\begin{equation} \label{quantum potential}
		Q(\varrho, \nabla_x \varrho, \nabla_x^2 \varrho)= \frac{\hbar}{2} \frac{\Delta_x \sqrt{\varrho}}{\sqrt{\varrho}},
	\end{equation}
	where $\varrho$ denotes the density of the fluid. More precisely, we are going to consider the following two models.
	\begin{itemize}
		\item The compressible \textit{quantum Navier--Stokes} system, whenever we are dealing with viscous fluids:
		\begin{align}
		\partial_t \varrho + \divv_x (\varrho\textbf{u} )&=0, \label{nsk continuity equation}\\
		\partial_t (\varrho \textbf{u}) + \divv_x \left( \varrho\textbf{u} \otimes \textbf{u} \right) + \nabla_x p(\varrho) &= \divv_x \mathbb{S}(\nabla_x \textbf{u})+  \varrho \nabla_x Q(\varrho, \nabla_x \varrho, \nabla_x^2 \varrho). \label{nsk balance of momentum}
		\end{align}
		\item The compressible \textit{quantum Euler} system, in case of inviscid fluids:
		\begin{align}
		\partial_t \varrho + \divv_x \textbf{J} &=0, \label{ek continuity equation}\\
		\partial_t \textbf{J} + \divv_x \left(\frac{\textbf{J} \otimes \textbf{J}}{\varrho}\right)  +  \nabla_x p(\varrho) &=\varrho \nabla_x Q(\varrho, \nabla_x \varrho, \nabla_x^2 \varrho). \label{ek balance of momentum}
		\end{align}
	\end{itemize}
	In both systems, the unknown variables are the density $\varrho=\varrho(t,x)$, the velocity $\textbf{u}=\textbf{u}(t,x)$ and the momentum $\textbf{J}=(\varrho \textbf{u})(t,x)$ of the fluid, while $p=p(\varrho)$ denotes the barotropic pressure, $\mathbb{S}=\mathbb{S}(\nabla_x \textbf{u})$ the viscous  stress tensor and $Q=Q(\varrho, \nabla_x \varrho, \nabla_x^2 \varrho)$ the quantum potential defined in \eqref{quantum potential}. More precisely, we will consider the standard isentropic pressure
	\begin{equation} \label{pressure}
	p(\varrho) = a\varrho^{\gamma}
	\end{equation}
	with $a$ a positive constant and $\gamma>\frac{d}{2}$ the adiabatic exponent, while the viscous stress tensor will be a linear function of the velocity gradient
	\begin{equation} \label{viscosity}
	\mathbb{S}(\nabla_x \textbf{u})= \mu \left(\nabla_x \textbf{u}+\nabla_x^{\top} \textbf{u}- \frac{2}{d}(\divv_x \textbf{u})\mathbb{I}\right) + \lambda(\divv_x \textbf{u})\mathbb{I}
	\end{equation}
	where $\mu >0$ and $\lambda \geq 0$ denote the shear and bulk viscosities, respectively.  Notice that we can write
	\begin{equation*}
	\varrho \nabla_x Q(\varrho, \nabla_x \varrho, \nabla_x^2 \varrho)= \divv_x \mathbb{K}(\varrho, \nabla_x \varrho, \nabla_x^2 \varrho)
	\end{equation*}
	with
	\begin{equation*}
	\mathbb{K}(\varrho, \nabla_x \varrho, \nabla_x^2 \varrho)  = \frac{\hbar}{4} \big( \nabla_x^2 \varrho- 4 \ \nabla_x \sqrt{\varrho} \otimes \nabla_x\sqrt{\varrho} \big).
	\end{equation*}
	
	We will study both systems on the set $(0,\infty) \times \Omega$, where $\Omega \subset \mathbb{R}^d$, $d=2,3$, is a bounded domain of class $C^2$, on the boundary of which we impose the homogeneous Neumann condition for the density and the no-slip condition for the velocity
	\begin{equation} \label{nsk boundary condition}
	\nabla_x \varrho \cdot \textbf{n}|_{\partial \Omega}=0, \quad \textbf{u}|_{\partial \Omega}=0,
	\end{equation}
	when considering system \eqref{nsk continuity equation}--\eqref{nsk balance of momentum}, while we replace the boundary condition for the velocity with the one for the momentum
	\begin{equation} \label{ek boundary conditions}
	\nabla_x \varrho \cdot \textbf{n}|_{\partial \Omega}=0, \quad \textbf{J} \cdot \textbf{n}|_{\partial \Omega}=0,
	\end{equation}
	when considering system \eqref{ek continuity equation}--\eqref{ek balance of momentum}.
	
	The last ingredient we need to formally close the systems is the energy. Introducing the \textit{pressure potential} $P=P(\varrho)$ as a solution of
	\begin{equation} \label{equation pressure potential}
		\varrho P'(\varrho)- P(\varrho) = p(\varrho),
	\end{equation}
	which we will consider as
	\begin{equation*}
		P(\varrho)= \frac{a}{\gamma-1} \varrho^{\gamma},
	\end{equation*}
	the \textit{total energy balance} associated to the quantum Navier--Stokes system \eqref{nsk continuity equation}--\eqref{nsk balance of momentum} with the boundary conditions \eqref{nsk boundary condition} is
	\begin{equation} \label{nsk total energy balance}
	\frac{\textup{d}}{\textup{d}t} \int_{\Omega} E(t) \ \textup{d}x + \int_{\Omega} \mathbb{S}(\nabla_x \textbf{u}): \nabla_x \textbf{u} \ \textup{d}x =0 \quad \mbox{with} \quad  E(t) = \frac{1}{2}\varrho |\textbf{u}|^2 +P(\varrho)+ \frac{\hbar}{2}  |\nabla_x \sqrt{\varrho}|^2,
	\end{equation}
	and similarly, the total energy balance associated to the quantum Euler system \eqref{ek continuity equation}--\eqref{ek balance of momentum} with the boundary conditions \eqref{ek boundary conditions} reads
	\begin{equation} \label{ek total energy balance}
		\frac{\textup{d}}{\textup{d}t} \int_{\Omega} E(t) \ \textup{d}x =0 \quad \mbox{with} \quad  E(t) = \frac{1}{2}\frac{|\textbf{J}|^2}{\varrho} +P(\varrho)+ \frac{\hbar}{2}  |\nabla_x \sqrt{\varrho}|^2.
	\end{equation}
	See Section \ref{Energies} for more details.
	
	We finally point out that the quantum potential $Q=Q(\varrho, \nabla_x \varrho, \nabla_x^2 \varrho)$ can be rewritten as
	\begin{equation} \label{Korteweg term}
		Q(\varrho, \nabla_x \varrho, \nabla_x^2 \varrho) = K(\varrho) \Delta_x \varrho + \frac{1}{2}K'(\varrho) |\nabla_x \varrho|^2,
	\end{equation}
	choosing the function $K=K(\varrho)$ such that
	\begin{equation*}
		K(\varrho)= \frac{\hbar}{4\varrho}.
	\end{equation*}
	Systems \eqref{nsk continuity equation}, \eqref{nsk balance of momentum} and \eqref{ek continuity equation}, \eqref{ek balance of momentum} with the quantum potential replaced by the more general expression appearing on the right-hand side of \eqref{Korteweg term} are called \textit{Navier--Stokes--Korteweg} and \textit{Euler--Korteweg} systems, respectively; usually, $K: (0,\infty) \rightarrow (0,\infty) $ is a smooth function.
	
	\subsection{State of the art}
	
	Given their importance in many applications, quantum fluid systems were widely studied in the last years. However, in literature we may typically encounter density-dependent instead of constant viscosity coefficients in the definition of the viscous stress tensor \eqref{viscosity}, leading to more mathematical difficulties due to the possible presence of vacuum. This alternative formulation is a consequence of a different derivation of the model, based on a Chapman--Enskog expansion of the Wigner function \cite{BruMeh}. For the quantum Navier--Stokes system with non-constant viscosity coefficients, the existence of global-in-time weak solutions with special test function $\rho\phi$ instead of classical test function $\phi$ on the $d$-dimensional torus, $d=2,3$, was shown by J\"{u}ngel \cite{Jun} with the constraint $\gamma>3$ for $d=3$ and the viscosity constant smaller than the scaled Plank constant; his result was later improved by Dong \cite{Don} and by Jiang \cite{Jia}, including the cases when the viscosity constant is equal and bigger, respectively, to the scaled Plank constant. Subsequently, the existence of global-in-time weak solutions with the standard test function $\phi$ was achieved with the help of extra terms in the equations that could guarantee the velocity to be well-defined even in the vacuum region: for instance, Gisclon and Lacroix-Violet \cite{GisLac} considered a cold pressure term, while Vasseur and Yu \cite{VasYu} added a damping term in the balance of momentum. Inspired by Li and Xin \cite{l},  Antonelli and Spirito \cite{AntSpi} proved the global-in-time existence result for weak solutions without any extra terms, but requiring the viscosity and capillarity constants to be comparable. Recently, this assumption was removed by the same authors in \cite{a3}, and by Lacroix-Violet and Vasseur \cite{LacVas}. Stability, i.e. the continuous dependence of solutions on initial data, was studied by Giesselmann, Lattanzio and Tzavaras \cite{GieLatTza} via a relative energy approach. Recently, Bresch, Gisclon and Lacroix-Violet \cite{BreGisLac} proved the existence of global-in-time dissipative solutions on the $d$-dimensional torus, $d=2,3$, of the quantum Navier--Stokes system with a linear density-dependent shear viscosity and zero bulk viscosity. Moreover, taking the vanishing viscosity limit, they obtained the existence of global-in-time dissipative solutions to the quantum Euler system \eqref{ek continuity equation}, \eqref{ek balance of momentum}. For the latter, there are several results concerning well-posedness in the class of weak solutions. Donatelli, Feireisl and Marcati \cite{DonFeiMar} showed that the system is ill-posed as uniqueness fails to be verified: for sufficiently smooth initial data, the system admits infinitely many weak solutions, even considering only the class of those satisfying the energy inequality. Later on, Antonelli and Marcati \cite{AntMar} proved the existence of global-in-time irrotational weak solutions by converting the Euler system into the non-linear Schr\"{o}dinger one, while Audiard and Haspot  \cite{AudHas} showed global well-posedness for small irrotational data in dimension $d\geq 3$. Last but not least, let us stress that important progresses have been made on singular limits and other topics for quantum fluid models \cite{a4,ca,c,d1,d2,Ju,li}.

	Even though there is a wide range of significant results concerning well-posedness of quantum  systems, we emphasize that there aren't any regarding the existence of global-in-time weak solutions for the quantum Navier--Stokes system \eqref{nsk continuity equation}, \eqref{nsk balance of momentum} with constant viscosity coefficients, even in dimension $d=2$, and for the quantum Euler system \eqref{ek continuity equation}, \eqref{ek balance of momentum} for large initial data, as pointed out by Bresch et al. \cite{BreGisLac}. Therefore, the latter are important and interesting issues.
	
	\subsection{Structure of the paper}
	
	In the present study, we are interested in well-posedness of the aforementioned quantum systems; specifically, we are concerned with existence and uniqueness of global-in-time solutions for any finite energy initial data. Inspired by the work of Abbatiello, Feireisl and Novotn\'{y} \cite{AbbFeiNov}, we will consider \textit{dissipative solutions}, i.e. solutions satisfying  the equations and the energy inequality in the distributional sense but with extra ``defect terms", which we may call \textit{Reynolds stresses}, collecting the possible oscillations and/or concentrations arising from the convective, pressure and quantum terms, cf. Definitions \ref{nsk dissipative solution} and \ref{ek dissipative solution}. This notion of solution can be seen as a generalization of the concept of dissipative measure--valued solution, developed by Feireisl, Gwiazda, \'{S}wierczewska-Gwiazda and Wiedemann \cite{FeiGwiSwiWie}, implying in particular that they can be taken into account in the analysis of convergence of certain numerical schemes and, therefore, they can be identified as strong limits of finite element--finite volume schemes in the spirit of Feireisl and Luk\'{a}\v{c}ov\'{a}--Medvi\softd ov\'{a} \cite{FeiLuk}. We point out that our definition of dissipative solution differs from the one considered in \cite{BreGisLac}, as the latter is based on a relative energy inequality. A natural question is whether strong solutions are uniquely determined in the class of dissipative solutions; in order to give a positive answer, we will prove the \textit{weak-strong uniqueness principle}: if the system admits a sufficiently regular solution in the classical sense then it must coincide with the dissipative solution emanating from the same initial data, cf. Theorems \ref{nsk weak-strong uniqueness} and \ref{ek weak-strong uniqueness}. As the name suggests, this technique was first developed by Prodi \cite{Pro} considering weak/strong solutions for the incompressible Navier--Stokes equations, and later adapted for compressible systems (see e.g. \cite{Fei}, \cite{FeiJinNov}, \cite{FeiNov}, \cite{Ger}, \cite{GwiSwiWie}, \cite{Wie}). Our next goal is the existence of dissipative solutions. More precisely, we will first prove the existence result for the quantum Navier--Stokes system \eqref{nsk continuity equation}, \eqref{nsk balance of momentum} applying the classical fixed point argument in the spirit of \cite{Fei1}, cf. Theorem \ref{nsk existence}, and afterwards, we will obtain the existence result for the quantum Euler system \eqref{ek continuity equation}, \eqref{ek balance of momentum} as a vanishing viscosity limit of the Navier--Stokes equations, cf. Theorem \ref{ek existence}. Finally, to handle the problem of uniqueness, especially in view of the ``negative" result stated in \cite{DonFeiMar} for the quantum Euler system, we may look for that particular dissipative solution in the class of the ones emanating from the same initial data satisfying the \textit{semigroup} or \textit{semiflow property}: if we let the system run from time $0$ to time $t_1$, we restart it and let it run for a time interval of amplitude $t_2$, the trajectory described by the selected solution will be the same as we have run the system directly from time $0$ to time $t_1+t_2$. We will refer to the process of finding such particular solution as \textit{semiflow selection}, cf. Definition \ref{semiflow selection}. Clearly, if uniqueness holds, the semigroup property is verified by any solution and the semiflow selection is simply the map associating to any admissible data that one unique solution emanating from it. The construction of a semiflow selection was originally a stochastic tool, first developed by Krylov \cite{Kry} to study well-posedness of certain systems and later adapted by Flandoli and Romito \cite{FlaRom}, Breit, Feireisl and Hofmanov\'{a} \cite{BreFeiHof1} for the incompressible and compressible, respectively, Navier--Stokes systems. Inspired by deterministic adaptation of Cardona and Kapitanski \cite{CarKap}, we will prove the existence of a semiflow selection for the quantum Navier--Stokes and quantum Euler systems in the class of dissipative solutions, cf. Theorems \ref{nsk semiflow selection} and \ref{ek semiflow selection}. We will essentially follow the same strategy developed by Breit, Feireisl and Hofmanov\'{a} \cite{BreFeiHof} for the compressible Euler system in the class of measure--valued solutions. However, there will be a slightly difference in the choice of the trajectory space: instead of the space of continuous functions as in \cite{CarKap} or the space of integrable functions as in \cite{BreFeiHof}, we will work with the Skorokhod space of c\`{a}gl\`{a}d (a French acronym for ``left-continuous and having right-hand limits") functions. The advantages of this choice is that on the one hand we are able to consider the energy, which is typically a non--increasing quantity with possible jumps, as a third state variable, while on the other hand we will get the existence of well-defined semiflow selections at any time. We point out that, thanks to the weak-strong uniqueness principle, solutions in the classical sense are always contained in the selected semiflow as long as they exist.
	
	\section{Dissipative solutions}
	
	In this section, we provide the definition of dissipative solution for both the quantum Navier--Stokes and quantum Euler systems. We will refer to the measure $\mathfrak{R}$ appearing in the weak formulations of the balance of momentum and energy inequality as \textit{Reynolds stress}. For the definition of all the involved spaces see Section \ref{Function spaces}.
	
	\begin{definition}[Dissipative solution of the quantum Navier--Stokes system] \label{nsk dissipative solution}
		The pair of functions $[\varrho,\textbf{u}]$ with total energy $E$ constitutes a \textit{dissipative solution} to problem \eqref{nsk continuity equation}--\eqref{nsk balance of momentum} with the isentropic pressure \eqref{pressure}, the viscous stress tensor \eqref{viscosity}, the boundary conditions \eqref{nsk boundary condition} and the initial data
		\begin{equation*}
		[\varrho(0,\cdot), (\varrho \textbf{u})(0,\cdot), E(0-)]=[\varrho_0, \textbf{J}_0, E_0] \in L^{\gamma}(\Omega)\times L^{\frac{2\gamma}{\gamma+1}}(\Omega; \mathbb{R}^d) \times [0,\infty)
		\end{equation*}
		if the following holds:
		\begin{itemize}
			\item[(i)]  \textit{regularity class}: $\varrho > 0$ in $(0,\infty) \times \Omega$ and
			\begin{align}
			\varrho &\in C_{\rm weak,loc}([0,\infty);  L^{\gamma}(\Omega)) \cap L^{\infty}(0,\infty; W^{1,\frac{2\gamma}{\gamma+1}}(\Omega)) \\[0.1cm]
			\varrho\textbf{u} &\in C_{\rm weak,loc}([0,\infty); L^q(\Omega;\mathbb{R}^d)), \quad q= \max \left\{ \frac{2\gamma}{\gamma+1},  \frac{4\gamma d}{(3d-2)\gamma+d}\right\}, \label{weak convergence momenta}\\[0.1cm]
			\textbf{u} &\in L^2_{\textup{loc}}(0,\infty; W^{1,2}_0(\Omega; \mathbb{R}^d)),\\[0.1cm]
			E &\in \mathfrak{D} ([0,\infty));
			\end{align}
			\item[(ii)] \textit{weak formulation of the continuity equation}: the integral identity
			\begin{equation} \label{nsk weak formulation continuity equation}
			\left[ \int_{\Omega} \varrho \varphi(t,\cdot) \ {\rm d}x \right]_{t=0}^{t=\tau}= \int_{0}^{\tau} \int_{\Omega} [\varrho \partial_t \varphi + \varrho\textbf{u}\cdot \nabla_x \varphi] \ {\rm d}x {\rm d}t
			\end{equation}
			holds for any $\tau >0$ and any $\varphi \in C^1_c([0,\infty)\times \overline{\Omega})$;
			\item[(iii)] \textit{weak formulation of the balance of momentum}: there exists
			\begin{equation*}
			\mathfrak{R} \in L^{\infty}(0,T; \mathcal{M}^+(\overline{\Omega}; \mathbb{R}^{d\times d }_{\rm sym}))
			\end{equation*}
			such that the integral identity
			\begin{equation} \label{nsk weak formulation balance of momentum}
			\begin{aligned}
			\left[\int_{\Omega}\varrho \textup{\textbf{u}} \cdot \bm{\varphi}(t, \cdot)\ \textup{d}x\right]_{t=0}^{t=\tau}  &= \int_{0}^{\tau}\int_{\Omega} \left[ \varrho\textup{\textbf{u}} \cdot \partial_t\bm{\varphi}+(\varrho \textup{\textbf{u}} \otimes \textup{\textbf{u}}) : \nabla_x \bm{\varphi} +p(\varrho)\divv_x \bm{\varphi} \right] \textup{d}x\textup{d}t \\
			&+ \frac{\hbar}{4}  \int_{0}^{\tau}\int_{\Omega} \big[ \nabla_x \varrho\cdot  \divv_x \nabla_x^{\top}\bm{\varphi} +4 (\nabla_x \sqrt{\varrho} \otimes \nabla_x \sqrt{\varrho}): \nabla_x \bm{\varphi} \big] \ \textup{d}x \textup{d}t \\
			&-\int_{0}^{\tau}\int_{\Omega}\mathbb{S} (\nabla_x \textup{\textbf{u}}): \nabla_x \bm{\varphi} \ \textup{d}x\textup{d}t + \int_{0}^{\tau}  \int_{\overline{\Omega}} \nabla_x \bm{\varphi} : \textup{d} \mathfrak{R} \  \textup{d}t
			\end{aligned}
			\end{equation}
			holds for any $\tau >0$ and any $\bm{\varphi} \in C^1_c([0,\infty); C^2(\overline{\Omega}; \mathbb{R}^d)) , \ \bm{\varphi}|_{\partial \Omega}=0$;
			\item[(iv)] \textit{energy inequality:} there exists a constant $\lambda>0$ such that
			\begin{equation*}
			\int_{\Omega} \left[ \frac{1}{2} \varrho |\textbf{\textup{u}}|^2 +P(\varrho) + \frac{\hbar}{2}  |\nabla_x \sqrt{\varrho}|^2\right](\tau, \cdot) \ \textup{d}x + \frac{1}{\lambda}\int_{\overline{\Omega}} \textup{d} \trace[\mathfrak{R}](\tau)= E(\tau)
			\end{equation*}
			for a.e. $\tau>0$, and the energy inequality
			\begin{equation} \label{nsk energy inequality}
			\big[ E(t)\psi(t) \big]_{t=\tau_1^-}^{t=\tau_2^+} - \int_{\tau_1}^{\tau_2} E \psi' \ \textup{d}t + \int_{\tau_1}^{\tau_2} \psi \int_{\Omega} \mathbb{S}(\nabla_x \textbf{u}): \nabla_x \textbf{u} \ \textup{d}x \textup{d}t \leq 0
			\end{equation}
			holds for any $0\leq \tau_1 \leq \tau_2$ and any $\psi \in C^1_c ([0,\infty))$, $\psi \geq 0$.
		\end{itemize}
	\end{definition}
	
	\begin{definition}[Dissipative solution of the quantum Euler system] \label{ek dissipative solution}
		The pair of functions $[\varrho,\textbf{J}]$ with total energy $E$ constitutes a \textit{dissipative solution} to problem \eqref{ek continuity equation}--\eqref{ek balance of momentum} with the isentropic pressure \eqref{pressure}, the boundary conditions \eqref{ek boundary conditions} and the initial data
		\begin{equation*}
		[\varrho(0,\cdot), \textbf{J}(0,\cdot), E(0-)]=[\varrho_0, \textbf{J}_0, E_0] \in L^{\gamma}(\Omega)\times L^{\frac{2\gamma}{\gamma+1}}(\Omega; \mathbb{R}^d) \times [0,\infty)
		\end{equation*}
		if the following holds:
		\begin{itemize}
			\item[(i)]  \textit{regularity class}: $\varrho > 0$ in $(0,\infty) \times \Omega$ and
			\begin{align*}
			\varrho &\in C_{\rm weak,loc}([0,\infty);  L^{\gamma} \cap W^{1,\frac{2\gamma}{\gamma+1}} (\Omega)) \\[0.1cm]
			\textbf{J} &\in C_{\rm weak,loc}([0,\infty); L^q(\Omega;\mathbb{R}^d)), \quad q= \max \left\{ \frac{2\gamma}{\gamma+1},  \frac{4\gamma d}{(3d-2)\gamma+d}\right\}, \\[0.1cm]
			E &\in \mathfrak{D} ([0,\infty));
			\end{align*}
			\item[(ii)] \textit{weak formulation of the continuity equation}: the integral identity
			\begin{equation} \label{ek weak formulation continuity equation}
			\left[ \int_{\Omega} \varrho \varphi(t,\cdot) \ {\rm d}x \right]_{t=0}^{t=\tau}= \int_{0}^{\tau} \int_{\Omega} [\varrho \partial_t \varphi + \textbf{J}\cdot \nabla_x \varphi] \ {\rm d}x {\rm d}t
			\end{equation}
			holds for any $\tau >0$ and any $\varphi \in C^1_c([0,\infty)\times \overline{\Omega})$;
			\item[(iii)] \textit{weak formulation of the balance of momentum}: there exists
			\begin{equation*}
			\mathfrak{R} \in L^{\infty}(0,T; \mathcal{M}^+(\overline{\Omega}; \mathbb{R}^{d\times d }_{\rm sym}))
			\end{equation*}
			such that the integral identity
			\begin{equation} \label{ek weak formulation balance of momentum}
			\begin{aligned}
			\left[\int_{\Omega} \textup{\textbf{J}} \cdot \bm{\varphi}(t, \cdot)\ \textup{d}x\right]_{t=0}^{t=\tau}  &= \int_{0}^{\tau}\int_{\Omega} \left[ \textup{\textbf{J}} \cdot \partial_t\bm{\varphi}+ \frac{\textup{\textbf{J}} \otimes \textup{\textbf{J}}}{\varrho}  : \nabla_x \bm{\varphi} +p(\varrho)\divv_x \bm{\varphi} \right] \textup{d}x\textup{d}t \\
			&+ \frac{\hbar}{4}  \int_{0}^{\tau}\int_{\Omega} \big[ \nabla_x \varrho\cdot  \divv_x \nabla_x^{\top}\bm{\varphi} +4 (\nabla_x \sqrt{\varrho} \otimes \nabla_x \sqrt{\varrho}): \nabla_x \bm{\varphi} \big] \ \textup{d}x \textup{d}t   \\
			& + \int_{0}^{\tau}  \int_{\overline{\Omega}} \nabla_x \bm{\varphi} : \textup{d} \mathfrak{R} \  \textup{d}t
			\end{aligned}
			\end{equation}
			holds for any $\tau >0$ and any $\bm{\varphi} \in C^1_c([0,\infty); C^2(\overline{\Omega}; \mathbb{R}^d)) , \ \bm{\varphi}|_{\partial \Omega}=0,$;
			\item[(iv)] \textit{energy inequality:} there exists a constant $\lambda>0$ such that
			\begin{equation*}
			\int_{\Omega} \left[ \frac{1}{2}  \frac{|\textbf{\textup{J}}|^2}{\varrho}  +P(\varrho) + \frac{\hbar}{2}  |\nabla_x \sqrt{\varrho}|^2\right](\tau, \cdot) \ \textup{d}x + \frac{1}{\lambda}\int_{\overline{\Omega}} \textup{d} \trace[\mathfrak{R}](\tau)= E(\tau)
			\end{equation*}
			for a.e. $\tau>0$, and the energy inequality
			\begin{equation} \label{ek energy inequality}
			\big[ E(t)\psi(t) \big]_{t=\tau_1^-}^{t=\tau_2^+} - \int_{\tau_1}^{\tau_2} E \psi' \ \textup{d}t \leq 0
			\end{equation}
			holds for any $0\leq \tau_1 \leq \tau_2$ and any $\psi \in C^1_c ([0,\infty))$, $\psi \geq 0$.
		\end{itemize}
	\end{definition}

	\section{Weak--strong uniqueness}
	
	In this section, our goal is to prove the \textit{weak--strong uniqueness principle}: if the quantum Navier--Stokes (or quantum Euler) system admits a sufficiently regular classical solution, then it must coincide with the dissipative solution emanating from the same initial data. Hereafter, let
	\begin{align*}
		p_1&:= \min \left\{ \frac{\gamma}{\gamma-1}, \frac{2d\gamma}{(d+2)\gamma-d} \right\}, \\[0.1cm]
		p_2 &:= \min \left\{ \frac{2d\gamma}{(d+2)\gamma -2d}, \frac{2d\gamma}{4\gamma-d} \right\}, \\[0.1cm]
		p_3 &:= \min \left\{ \frac{d\gamma}{2\gamma-d}, \frac{2d\gamma}{(6-d)\gamma-d} \right\}.
	\end{align*}
	
	\begin{theorem}[Weak--strong uniqueness for the quantum Navier--Stokes system] \label{nsk weak-strong uniqueness}
		Let $[\widetilde{\varrho}, \widetilde{\textbf{\textup{u}}}]$ with $\widetilde{\varrho}>0$ and
		\begin{equation} \label{rc 1}
		\begin{aligned}
			\widetilde{\varrho} &\in L^{\infty}(0,\infty;  L^{\gamma} \cap  W^{1,\frac{2\gamma}{\gamma+1}} (\Omega)), \\[0.1cm]
			\widetilde{\textbf{\textup{u}}} &\in L^{\infty}(0,\infty; L^{2p_1}(\Omega; \mathbb{R}^d)) +L^2_{\textup{loc}}(0,\infty; W^{1,2}_0(\Omega; \mathbb{R}^d)),
		\end{aligned}
		\end{equation}\\[0.1cm]
		be a strong solution of system \eqref{nsk continuity equation}--\eqref{nsk balance of momentum}, satisfying the constitutive relations \eqref{pressure}--\eqref{viscosity} and the boundary conditions \eqref{nsk boundary condition}, where in addition the density $\widetilde{\varrho}$ is such that\\[0.1cm]
		\begin{equation} \label{rc 2}
			\begin{aligned}
				\partial_t P'(\widetilde{\varrho}) & \in L^1_{\rm loc}(0,\infty; L^{p_1}(\Omega)), \\[0.1cm]
				\nabla_x P'(\widetilde{\varrho}) & \in L^2_{\rm loc}(0,\infty; L^{p_2} (\Omega; \mathbb{R}^d))+ L^1_{\rm loc}(0,\infty; L^{2p_1}(\Omega; \mathbb{R}^d)),  \\[0.1cm]
				\partial_t \nabla_x \log \widetilde{\varrho} &\in L^1_{\rm loc}(0,\infty; L^{\frac{2\gamma}{\gamma-1}}(\Omega; \mathbb{R}^d)), \\[0.1cm]
				\nabla_x^ 2 \log \widetilde{\varrho} &\in L^2_{\rm loc}(0,\infty; L^{\frac{2d\gamma}{2\gamma-d}}(\Omega; \mathbb{R}^{d\times d}))+ L^1_{\rm loc}(0,\infty; L^{\infty}(\Omega; \mathbb{R}^{d\times d})),
			\end{aligned}
		\end{equation}\\[0.1cm]
		the velocity $\widetilde{\textbf{\textup{u}}}$ is such that
		\begin{equation} \label{rc 3}
			\begin{aligned}
				\partial_t \widetilde{\textbf{\textup{u}}} & \in L^2_{\rm loc}(0,\infty; L^{p_2 }(\Omega; \mathbb{R}^d))+ L^1_{\rm loc}(0,\infty; L^{2p_1}(\Omega; \mathbb{R}^d)),  \\[0.1cm]
				\nabla_x \widetilde{\textbf{\textup{u}}} & \in L^{\infty}(0,\infty; L^{p_3}(\Omega; \mathbb{R}^{d\times d}))+ L^2_{\rm loc}(0,\infty; L^{2p_3}(\Omega; \mathbb{R}^{d\times d}))\\
				&+ L^1_{\rm loc}(0,\infty; L^{\infty}(\Omega; \mathbb{R}^{d\times d})),  \\[0.1cm]
				\divv_x \widetilde{\textbf{\textup{u}}} &\in L^1_{\rm loc} (0,\infty; L^{\infty}(\Omega)),  \\[0.1cm]
				\divv_x \nabla_x^{\top} \widetilde{\textbf{\textup{u}}} &\in L^1_{\rm loc}(0,\infty; L^{\frac{2\gamma}{\gamma-1}}(\Omega; \mathbb{R}^d)),  \\[0.1cm]
			\end{aligned}
		\end{equation}
		and
		\begin{equation*}
			\frac{\mathbb{S}(\nabla_x \widetilde{\textbf{\textup{u}}})}{\widetilde{\varrho}} \in L^2_{\rm loc}(0,\infty; L^{\frac{2d\gamma}{2\gamma-d}}(\Omega; \mathbb{R}^{d\times d})).
		\end{equation*}
		
		Let $[\varrho, \textbf{\textup{u}}]$ be a dissipative solution of the same system with dissipation defect $\mathfrak{R}$ in the sense of Definition \ref{nsk dissipative solution}. If
		\begin{equation} \label{nsk same initial conditions}
		[\widetilde{\varrho}(0,x), (\widetilde{\varrho}\widetilde{\textbf{\textup{u}}})(0,x) ] = [\varrho(0,x), (\varrho\textbf{\textup{u}})(0,x) ]  \quad \mbox{for a.e. } x\in \Omega
		\end{equation}
		then $\mathfrak{R} \equiv 0$ and
		\begin{equation} \label{nsk same solutions}
		[\widetilde{\varrho}(t,x), \widetilde{\textbf{\textup{u}}}(t,x) ] = [\varrho(t,x), \textbf{\textup{u}}(t,x) ]  \quad \mbox{for a.e. } (t,x)\in (0,\infty) \times  \Omega.
		\end{equation}
	\end{theorem}
	
	\begin{theorem}[Weak--strong uniqueness for the quantum Euler system] \label{ek weak-strong uniqueness}
		Let $[\widetilde{\varrho}, \widetilde{\textbf{\textup{u}}}]$ with
		\begin{equation}
		\begin{aligned}
		\widetilde{\varrho} &\in L^{\infty}(0,\infty;  L^{\gamma} \cap  W^{1,\frac{2\gamma}{\gamma+1}} (\Omega)), \\[0.1cm]
		\widetilde{\textbf{\textup{u}}} &\in L^{\infty}(0,\infty; L^{2p_1}(\Omega; \mathbb{R}^d)),
		\end{aligned}
		\end{equation}
		be a strong solution of system \eqref{ek continuity equation}--\eqref{ek balance of momentum} satisfying the constitutive relation \eqref{pressure}, where in addition the density $\widetilde{\varrho}>0$ and the velocity $\widetilde{\textbf{\textup{u}}}$ are such that $\widetilde{\textbf{\textup{u}}}\cdot \textbf{\textup{n}}|_{\partial \Omega}=0$ and\\[0.1cm]
		\begin{equation}
		\begin{aligned}
		\partial_t P'(\widetilde{\varrho}) & \in L^1_{\rm loc}(0,\infty; L^{p_1}(\Omega)), \\[0.1cm]
		\divv_x \widetilde{\textbf{\textup{u}}} &\in L^1_{\rm loc} (0,\infty; L^{\infty}(\Omega)),  \\[0.1cm]
		\nabla_x P'(\widetilde{\varrho}), \ \nabla_x \Delta_x \log \widetilde{\varrho}, \ \partial_t \widetilde{\textbf{\textup{u}}} & \in L^1_{\rm loc}(0,\infty; L^{2p_1}(\Omega; \mathbb{R}^d)),  \\[0.1cm]
		\partial_t \nabla_x \log \widetilde{\varrho}, \ 	\divv_x \nabla_x^{\top} \widetilde{\textbf{\textup{u}}} &\in L^1_{\rm loc}(0,\infty; L^{\frac{2\gamma}{\gamma-1}}(\Omega; \mathbb{R}^d)), \\[0.1cm]
		\nabla_x^ 2 \log \widetilde{\varrho}, \ \nabla_x \widetilde{\textbf{\textup{u}}}  &\in  L^1_{\rm loc}(0,\infty; L^{\infty}(\Omega; \mathbb{R}^{d\times d})).
		\end{aligned}
		\end{equation}\\[0.1cm]
		Let $[\varrho, \textbf{\textup{J}}]$ be a dissipative solution of the same system with dissipation defect $\mathfrak{R}$ in the sense of Definition \ref{ek dissipative solution}. If
		\begin{equation*}
		[\widetilde{\varrho}(0,x), (\widetilde{\varrho}\widetilde{\textbf{\textup{u}}})(0,x) ] = [\varrho(0,x), \textbf{\textup{J}}(0,x) ]  \quad \mbox{for a.e. } x\in \Omega
		\end{equation*}
		then $\mathfrak{R} \equiv 0$ and
		\begin{equation*}
		[\widetilde{\varrho}(t,x), (\widetilde{\varrho}\widetilde{\textbf{\textup{u}}})(t,x) ] = [\varrho(t,x), \textbf{\textup{J}}(t,x) ]  \quad \mbox{for a.e. } (t,x)\in (0,\infty) \times  \Omega.
		\end{equation*}
	\end{theorem}

	The proofs are based on showing that a slightly modified version of the energy, known as relative energy, and the Reynolds stress vanish almost everywhere.
	
	\subsection{Proof of Theorem \ref{nsk weak-strong uniqueness}}
	
	We introduce the \textit{relative energy functional}:
	\begin{equation*}
		E (\varrho, \nabla_x \varrho, \textbf{u} \ | \ \widetilde{\varrho}, \nabla_x \widetilde{\varrho}, \widetilde{\textbf{\textup{u}}})=\frac{1}{2} \varrho |\textbf{u}-\widetilde{\textbf{\textup{u}}}|^2  + P(\varrho) -P'(\widetilde{\varrho}) (\varrho-\widetilde{\varrho}) -P(\widetilde{\varrho}) + \frac{\hbar}{2}  \left| \nabla_x \sqrt{\varrho}- \sqrt{\frac{\varrho}{\widetilde{\varrho}}} \ \nabla_x\sqrt{\widetilde{\varrho}}\right|^2.
	\end{equation*}
	To simplify notation, we introduce the \textit{drift velocities}
	\begin{equation} \label{drift velocity}
		\textbf{v}= \frac{\nabla_x \sqrt{\varrho}}{\sqrt{\varrho}}, \quad \widetilde{\textbf{v}} = \frac{ \nabla_x\sqrt{\widetilde{\varrho}}}{\sqrt{\widetilde{\varrho}}},
	\end{equation}
	and therefore the relative energy functional can be rewritten as
		\begin{equation*}
	E (\varrho, \textbf{u}, \textbf{v} \ | \ \widetilde{\varrho}, \widetilde{\textbf{\textup{u}}}, \widetilde{\textbf{v}})=\frac{1}{2} \varrho |\textbf{u}-\widetilde{\textbf{\textup{u}}}|^2  + P(\varrho) -P'(\widetilde{\varrho}) (\varrho-\widetilde{\varrho}) -P(\widetilde{\varrho}) + \frac{\hbar}{2} \varrho \left|\textbf{v}-\widetilde{\textbf{v}}\right|^2.
	\end{equation*}\\[0.001cm]

	\textbf{Step 1.} First of all, proving Theorem \ref{nsk weak-strong uniqueness} is equivalent in showing that
	\begin{equation} \label{equivalent condition}
	\mathfrak{R} \equiv 0, \quad E(\varrho, \textbf{u}, \textbf{v} \ | \ \widetilde{\varrho}, \widetilde{\textbf{\textup{u}}}, \widetilde{\textbf{v}}) \equiv 0 \quad \mbox{a.e. in } (0,\infty) \times \Omega.
	\end{equation}
	Indeed, since the pressure $\varrho \mapsto p(\varrho)$ is strictly increasing in $(0,\infty)$, the pressure potential $\varrho \mapsto P(\varrho)$ is strictly convex. For a differentiable function, this is equivalent in saying that the function lies above all of its tangents,
	\begin{equation} \label{relation pressure potential}
	P(\varrho) \geq P'(\widetilde{\varrho})  (\varrho-\widetilde{\varrho})+ P(\widetilde{\varrho})
	\end{equation}
	for all $\varrho, \widetilde{\varrho} \in (0,\infty)$. Therefore, we can deduce that
	\begin{equation} \label{non-negativity relative energy}
	E(\varrho, \textbf{u}, \textbf{v} \ | \ \widetilde{\varrho}, \widetilde{\textbf{\textup{u}}}, \widetilde{\textbf{v}}) \geq 0.
	\end{equation}
	Moreover, the equality in \eqref{relation pressure potential} holds if and only if $\varrho= \widetilde{\varrho}$ and consequently the equality in \eqref{non-negativity relative energy} holds if and only if \eqref{nsk same solutions} holds. \\[0.05cm]
	
	\textbf{Step 2.} We will now show that any dissipative solution satisfies an extended version of the energy inequality, whenever $[\widetilde{\varrho}, \widetilde{\textbf{\textup{u}}}]$ are smooth and compactly supported functions. Let us suppose that
	\begin{align*}
	\widetilde{\varrho} &\in C_c^{\infty} ([0,\infty)\times \overline{\Omega}), \\
	\widetilde{\textbf{u}} &\in C_c^{\infty} ([0,\infty)\times \overline{\Omega}; \mathbb{R}^d);
	\end{align*}
	then, we can take $\varphi= \frac{1}{2} |\widetilde{\textbf{u}}|^2, \ P'(\widetilde{\varrho}), \ \frac{\hbar}{2} |\widetilde{\textbf{v}}|^2, \ \hbar \divv_x \widetilde{\textbf{v}}$ as test functions in the weak formulation of the continuity equation \eqref{nsk weak formulation continuity equation} to get
	\begin{align}
	\frac{1}{2}\left[ \int_{\Omega} \varrho |\widetilde{\textbf{u}}|^2(t,\cdot) \ {\rm d}x \right]_{t=0}^{t=\tau}&= \int_{0}^{\tau} \int_{\Omega} \varrho \widetilde{\textbf{u}} \cdot \big( \partial_t \widetilde{\textbf{u}} + \nabla_x \widetilde{\textbf{u}} \cdot \textbf{u} \big)   \ {\rm d}x {\rm d}t, \label{e1}\\[0.1cm]
	\left[ \int_{\Omega} \varrho P'(\widetilde{\varrho})(t,\cdot) \ {\rm d}x \right]_{t=0}^{t=\tau}&= \int_{0}^{\tau} \int_{\Omega} [\varrho \partial_t P'(\widetilde{\varrho}) + \varrho\textbf{u}\cdot \nabla_x P'(\widetilde{\varrho})] \ {\rm d}x {\rm d}t, \label{e2}\\[0.1cm]
	\frac{\hbar}{2}\left[ \int_{\Omega} \varrho  |\widetilde{\textbf{v}}|^2 (t,\cdot) \ {\rm d}x \right]_{t=0}^{t=\tau}& = \hbar \int_{0}^{\tau} \int_{\Omega} \varrho \widetilde{\textbf{v}} \cdot \big( \partial_t \widetilde{\textbf{v}} + \nabla_x \widetilde{\textbf{v}} \cdot \textbf{u}  \big) \ {\rm d}x {\rm d}t, \label{e3} \\
	\hbar \left[ \int_{\Omega}  \varrho \textbf{v} \cdot \widetilde{\textbf{v}} (t,\cdot) \ {\rm d}x \right]_{t=0}^{t=\tau} &= \hbar \int_{0}^{\tau} \int_{\Omega} \big[\varrho \textbf{v} \cdot \big( \partial_t \widetilde{\textbf{v}} + \nabla_x \widetilde{\textbf{v}} \cdot \textbf{u}  \big) + \varrho \nabla_x \textbf{u} : \nabla_x \widetilde{\textbf{v}}\big]\  {\rm d}x {\rm d}t, \label{e4}
	\end{align}
	where we recall identity \eqref{identity}, and $\bm{\varphi}= \widetilde{\textbf{u}}$ as test function in the weak formulation of the balance of momentum \eqref{nsk weak formulation balance of momentum} to get
	\begin{equation} \label{e5}
	\begin{aligned}
	\left[\int_{\Omega}\varrho \textup{\textbf{u}} \cdot \widetilde{\textbf{u}}(t, \cdot)\ \textup{d}x\right]_{t=0}^{t=\tau}  &= \int_{0}^{\tau}\int_{\Omega} \left[ \varrho\textup{\textbf{u}} \cdot \big( \partial_t \widetilde{\textbf{u}} + \nabla_x \widetilde{\textbf{u}} \cdot \textbf{u} \big) +p(\varrho)\divv_x \widetilde{\textbf{u}}\right] \textup{d}x\textup{d}t \\
	&+ \hbar \int_{0}^{\tau}\int_{\Omega} \left[ \frac{1}{2}\varrho \textbf{v} \cdot   \divv_x \nabla_x^{\top} \widetilde{\textbf{u}} + \varrho \textbf{v} \cdot \nabla_x \widetilde{\textbf{u}} \cdot \textbf{v}\right] \textup{d}x \textup{d}t  \\
	&-\int_{0}^{\tau}\int_{\Omega}\mathbb{S} (\nabla_x \textup{\textbf{u}}): \nabla_x \widetilde{\textbf{u}} \ \textup{d}x\textup{d}t + \int_{0}^{\tau}  \int_{\overline{\Omega}} \nabla_x \widetilde{\textbf{u}} : \textup{d} \mathfrak{R} \  \textup{d}t.
	\end{aligned}
	\end{equation}
	Next, if we sum the integral identities \eqref{e1}, \eqref{e3} and subtract \eqref{e2}, \eqref{e4}, \eqref{e5} from the energy inequality \eqref{nsk energy inequality}, keeping in mind that
	\begin{equation*}
	\left[ \int_{\Omega} \big[ \widetilde{\varrho} P'(\widetilde{\varrho})- P(\widetilde{\varrho})\big] (t,\cdot) \ \textup{d}x\right]_{t=0}^{t=\tau} = \int_{0}^{\tau} \int_{\Omega} \frac{\partial}{\partial t} \big[ \widetilde{\varrho} P'(\widetilde{\varrho})- P(\widetilde{\varrho})\big] \ \textup{d}x \textup{d}t = \int_{0}^{\tau} \int_{\Omega} \widetilde{\varrho} P''(\widetilde{\varrho}) \partial_t \widetilde{\varrho} \ \textup{d}x \textup{d}t,
	\end{equation*}
	we obtain
	\begin{align*}
	&\left[ \int_{\Omega} E(\varrho, \textup{\textbf{u}}, \textbf{v} \ | \ \widetilde{\varrho}, \widetilde{\textup{\textbf{u}}}, \widetilde{\textbf{v}}) (t, \cdot) \  \textup{d}x\right]_{t=0}^{t=\tau} + \frac{1}{\lambda}\int_{\overline{\Omega}} \textup{d} \trace[\mathfrak{R}](\tau) + \int_{0}^{\tau} \int_{\Omega} \mathbb{S}(\nabla_x \textup{\textbf{u}}):\nabla_x (\textup{\textbf{u}}-\widetilde{\textup{\textbf{u}}}) \ \textup{d}x \textup{d}t \\
	&\hspace{1.5cm}\leq - \int_{0}^{\tau} \int_{\Omega}  \varrho(\textup{\textbf{u}}-\widetilde{\textup{\textbf{u}}})  \cdot[\partial_t \widetilde{\textup{\textbf{u}}} + \nabla_x \widetilde{\textup{\textbf{u}}} \cdot \widetilde{\textup{\textbf{u}}} + \nabla_x \widetilde{\textup{\textbf{u}}} \cdot ( \textup{\textbf{u}}-\widetilde{\textup{\textbf{u}}})] \ \textup{d}x \textup{d}t \\
	&\hspace{1.5cm}- \int_{0}^{\tau} \int_{\Omega} p(\varrho) \divv_x \widetilde{\textbf{u}} \ \textup{d}x\textup{d}t \\
	&\hspace{1.5cm} -\hbar\int_{0}^{\tau}\int_{\Omega} \varrho (\textbf{v}-\widetilde{\textbf{v}}) \cdot \big[ \partial_t \widetilde{\textbf{v}} + \nabla_x \widetilde{\textbf{v}} \cdot \widetilde{\textbf{u}} + \nabla_x \widetilde{\textbf{v}} \cdot (\textbf{u} -\widetilde{\textbf{u}}) \big] \  \textup{d}x \textup{d}t \\
	&\hspace{1.5cm} -\int_{0}^{\tau}  \int_{\overline{\Omega}} \nabla_x \widetilde{\textbf{u}} : \textup{d} \mathfrak{R} \  \textup{d}t -\int_{0}^{\tau} F_1(t) \ \textup{d}t - \hbar \int_{0}^{\tau} F_2(t) \ \textup{d}t,
	\end{align*}
	with
	\begin{align*}
	F_1 (t)&= \int_{\Omega} \frac{\varrho}{\widetilde{\varrho}} \  p'(\widetilde{\varrho}) \big( \partial_t \widetilde{\varrho} + \textbf{u}\cdot \nabla_x \widetilde{\varrho}\big)(t,\cdot) \ {\rm d}x- \int_{\Omega}  p'(\widetilde{\varrho}) \partial_t \widetilde{\varrho}(t,\cdot) \ \textup{d}x , \\
	F_2(t)&= \int_{\Omega} \varrho  \left( \frac{1}{2} \ \textbf{v} \cdot \divv_x \nabla_x^{\top} \widetilde{\textbf{u}} + \textbf{v} \cdot\nabla_x \widetilde{\textbf{u}} \cdot \textbf{v} + \frac{1}{2}\nabla_x \textbf{u} : \nabla_x \widetilde{\textbf{v}} \right)(t,\cdot) \ \textup{d}x,
	\end{align*}
	recalling that $p'(\widetilde{\varrho})= \widetilde{\varrho} P''(\widetilde{\varrho})$. Now, we can sum and subtract the following integrals
	\begin{equation} \label{e6}
	\int_{0}^{\tau} \int_{\Omega}  \varrho(\textup{\textbf{u}}-\widetilde{\textup{\textbf{u}}})  \cdot \left[ \nabla_x P'(\widetilde{\varrho})- \frac{1}{\widetilde{\varrho}} \divv_x \mathbb{S}(\nabla_x \widetilde{\textup{\textbf{u}}}) - \frac{1}{\widetilde{\varrho}} \divv_x \mathbb{K}(\widetilde{\varrho}, \nabla_x \widetilde{\textbf{v}}) \right]  \textup{d}x \textup{d}t,
	\end{equation}
	\begin{equation} \label{e7}
	\int_{0}^{\tau} \int_{\Omega}  [p'(\widetilde{\varrho})(\varrho-\widetilde{\varrho})+p(\widetilde{\varrho})] \divv_x \widetilde{\textup{\textbf{u}}} \ \textup{d}x \textup{d}t,
	\end{equation}
	\begin{equation} \label{e8}
	\frac{\hbar}{2} \int_{0}^{\tau} \int_{\Omega} \frac{\varrho}{\widetilde{\varrho}} \ (\textbf{v}-\widetilde{\textbf{v}}) \cdot \divv_x \big( \widetilde{\varrho} \nabla_x^{\top} \widetilde{\textbf{u}}\big) \ \textup{d}x \textup{d}t
	\end{equation}
	from the previous inequality to get
	\begin{align*}
	&\left[ \int_{\Omega} E(\varrho, \textup{\textbf{u}}, \textbf{v} \ | \ \widetilde{\varrho}, \widetilde{\textup{\textbf{u}}}, \widetilde{\textbf{v}}) (t, \cdot) \  \textup{d}x\right]_{t=0}^{t=\tau} + \frac{1}{\lambda}\int_{\overline{\Omega}} \textup{d} \trace[\mathfrak{R}](\tau) \\
	&\hspace{1.5cm}+ \int_{0}^{\tau} \int_{\Omega} \mathbb{S}\big(\nabla_x (\textup{\textbf{u}}- \widetilde{\textbf{u}})\big):\nabla_x (\textup{\textbf{u}}-\widetilde{\textup{\textbf{u}}}) \ \textup{d}x \textup{d}t \\
	&\hspace{1.1cm}\leq -\int_{0}^{\tau} \int_{\Omega}  \varrho(\textup{\textbf{u}}-\widetilde{\textup{\textbf{u}}})  \cdot \left[\partial_t \widetilde{\textup{\textbf{u}}} + \nabla_x \widetilde{\textup{\textbf{u}}} \cdot \widetilde{\textup{\textbf{u}}}+ \nabla_x P'(\widetilde{\varrho})- \frac{1}{\widetilde{\varrho}} \divv_x \mathbb{S}(\nabla_x \widetilde{\textup{\textbf{u}}}) - \frac{1}{\widetilde{\varrho}} \divv_x \mathbb{K}(\widetilde{\varrho}, \nabla_x \widetilde{\textbf{v}}) \right]  \textup{d}x \textup{d}t \\
	&\hspace{1.5cm}- \int_{0}^{\tau} \int_{\Omega} \varrho (\textup{\textbf{u}}-\widetilde{\textup{\textbf{u}}}) \cdot \nabla_x \widetilde{\textup{\textbf{u}}} \cdot ( \textup{\textbf{u}}-\widetilde{\textup{\textbf{u}}}) \ \textup{d}x \textup{d}t \\
	&\hspace{1.5cm} -\int_{0}^{\tau} \int_{\Omega}  [p(\varrho)- p'(\widetilde{\varrho})(\varrho-\widetilde{\varrho})-p(\widetilde{\varrho})] \divv_x \widetilde{\textup{\textbf{u}}} \ \textup{d}x \textup{d}t, \\
	&\hspace{1.5cm} - \int_{0}^{\tau}\int_{\Omega} \left( \frac{\varrho}{\widetilde{\varrho}} -1\right) (\textbf{u} -\widetilde{\textbf{u}}) \cdot \divv_x \mathbb{S}(\nabla_x \widetilde{\textbf{u}}) \ \textup{d}x \textup{d}t \\
	&\hspace{1.5cm} -\hbar \int_{0}^{\tau}\int_{\Omega} \varrho (\textbf{v}-\widetilde{\textbf{v}}) \cdot \left[ \partial_t \widetilde{\textbf{v}} + \nabla_x \widetilde{\textbf{v}} \cdot \widetilde{\textbf{u}} + \frac{1}{2 \widetilde{\varrho}} \divv_x \big( \widetilde{\varrho} \nabla_x^{\top} \widetilde{\textbf{u}} \big) \right]  \textup{d}x \textup{d}t \\
	&\hspace{1.5cm} -\hbar \int_{0}^{\tau}\int_{\Omega} \varrho (\textbf{v}-\widetilde{\textbf{v}}) \cdot  \nabla_x \widetilde{\textbf{v}} \cdot (\textbf{u} -\widetilde{\textbf{u}})  \  \textup{d}x \textup{d}t \\
	&\hspace{1.5cm} -\int_{0}^{\tau}  \int_{\overline{\Omega}} \nabla_x \widetilde{\textbf{u}} : \textup{d} \mathfrak{R} \  \textup{d}t -\int_{0}^{\tau} \widetilde{F}_1(t) \ \textup{d}t - \hbar \int_{0}^{\tau} \widetilde{F}_2(t) \ \textup{d}t,
	\end{align*}
	with
	\begin{align*}
	\widetilde{F}_1(t) = F_1(t) &- \int_{\Omega} \frac{\varrho}{\widetilde{\varrho}} \  p'(\widetilde{\varrho}) \big[\textbf{u} \cdot \nabla_x \widetilde{\varrho} - \divv_x \big( \widetilde{\varrho} \widetilde{\textbf{u}}\big)\big](t,\cdot) \ \textup{d}x - \int_{\Omega} p'(\widetilde{\varrho}) \divv_x \big( \widetilde{\varrho}\widetilde{\textbf{u}} \big)(t,\cdot) \ \textup{d}x \\
	&= \int_{\Omega} p'(\widetilde{\varrho}) \left(  \frac{\varrho}{\widetilde{\varrho}}- 1\right) [\partial_t \widetilde{\varrho} + \divv_x (\widetilde{\varrho} \widetilde{\textup{\textbf{u}}})](t,\cdot)  \  \textup{d}x,
	\end{align*}
	and
	\begin{align*}
	\widetilde{F}_2(t)= F_2(t) &+ \frac{1}{\hbar} \int_{\Omega} \frac{\varrho}{\widetilde{\varrho}} \  (\textbf{u}-\widetilde{\textbf{u}} ) \cdot \divv_x \mathbb{K}(\widetilde{\varrho}, \nabla_x\widetilde{\textbf{v}}) \  \textup{d}x - \frac{1}{2}\int_{\Omega} \frac{\varrho}{\widetilde{\varrho}} \  (\textbf{v}-\widetilde{\textbf{v}}) \cdot \divv_x \big( \widetilde{\varrho} \nabla_x^{\top} \widetilde{\textbf{u}} \big) \ \textup{d}x \\
	& = - \int_{\Omega} \varrho (\textbf{v}-\widetilde{\textbf{v}}) \cdot  \nabla_x \widetilde{\textbf{v}} \cdot (\textbf{u} -\widetilde{\textbf{u}})  \  \textup{d}x + \int_{\Omega} \varrho (\textbf{v}-\widetilde{\textbf{v}}) \cdot  \nabla_x \widetilde{\textbf{u}} \cdot (\textbf{v} -\widetilde{\textbf{v}})  \  \textup{d}x
	\end{align*}
	
	We have finally obtained the \textit{relative energy inequality}:
	\begin{equation} \label{relative energy inequality}
	\begin{aligned}
	&\left[ \int_{\Omega} E(\varrho, \textup{\textbf{u}}, \textbf{v} \ | \ \widetilde{\varrho}, \widetilde{\textup{\textbf{u}}}, \widetilde{\textbf{v}}) (t, \cdot) \  \textup{d}x\right]_{t=0}^{t=\tau} + \frac{1}{\lambda}\int_{\overline{\Omega}} \textup{d} \trace[\mathfrak{R}](\tau) \\
	&\hspace{0.4cm}+ \int_{0}^{\tau} \int_{\Omega} \mathbb{S}\big(\nabla_x (\textup{\textbf{u}}- \widetilde{\textbf{u}})\big):\nabla_x (\textup{\textbf{u}}-\widetilde{\textup{\textbf{u}}}) \ \textup{d}x \textup{d}t \\
	&\leq -\int_{0}^{\tau} \int_{\Omega}  \varrho(\textup{\textbf{u}}-\widetilde{\textup{\textbf{u}}})  \cdot \left[\partial_t \widetilde{\textup{\textbf{u}}} + \nabla_x \widetilde{\textup{\textbf{u}}} \cdot \widetilde{\textup{\textbf{u}}}+ \nabla_x P'(\widetilde{\varrho})- \frac{1}{\widetilde{\varrho}} \divv_x \mathbb{S}(\nabla_x \widetilde{\textup{\textbf{u}}}) - \frac{1}{\widetilde{\varrho}} \divv_x \mathbb{K}(\widetilde{\varrho}, \nabla_x\widetilde{\textbf{v}}) \right]  \textup{d}x \textup{d}t \\
	&\hspace{0.4cm}- \int_{0}^{\tau} \int_{\Omega} \varrho \big[(\textup{\textbf{u}}-\widetilde{\textup{\textbf{u}}}) \otimes (\textup{\textbf{u}}-\widetilde{\textup{\textbf{u}}})\big] : \nabla_x \widetilde{\textup{\textbf{u}}} \ \textup{d}x \textup{d}t \\
	&\hspace{0.4cm} -\int_{0}^{\tau} \int_{\Omega}  [p(\varrho)- p'(\widetilde{\varrho})(\varrho-\widetilde{\varrho})-p(\widetilde{\varrho})] \divv_x \widetilde{\textup{\textbf{u}}} \ \textup{d}x \textup{d}t \\
	&\hspace{0.4cm} -\int_{0}^{\tau} \int_{\Omega} p'(\widetilde{\varrho}) \left(  \frac{\varrho}{\widetilde{\varrho}}- 1\right) [\partial_t \widetilde{\varrho} + \divv_x (\widetilde{\varrho} \widetilde{\textup{\textbf{u}}})] \  \textup{d}x \textup{d}t, \\
	&\hspace{0.4cm} - \int_{0}^{\tau}\int_{\Omega} \left( \frac{\varrho}{\widetilde{\varrho}} -1\right) (\textbf{u} -\widetilde{\textbf{u}}) \cdot \divv_x \mathbb{S}(\nabla_x \widetilde{\textbf{u}}) \ \textup{d}x \textup{d}t \\
	&\hspace{0.4cm} -\hbar \int_{0}^{\tau}\int_{\Omega} \varrho (\textbf{v}-\widetilde{\textbf{v}}) \cdot \left[ \partial_t \widetilde{\textbf{v}} + \nabla_x \widetilde{\textbf{v}} \cdot \widetilde{\textbf{u}} + \frac{1}{2\widetilde{\varrho}} \divv_x \big( \widetilde{\varrho} \nabla_x^{\top} \widetilde{\textbf{u}} \big) \right]  \textup{d}x \textup{d}t \\
	&\hspace{0.4cm} -\hbar \int_{0}^{\tau}\int_{\Omega} \varrho \big[(\textup{\textbf{v}}-\widetilde{\textup{\textbf{v}}}) \otimes (\textup{\textbf{v}}-\widetilde{\textup{\textbf{v}}})\big] : \nabla_x \widetilde{\textup{\textbf{u}}} \  \textup{d}x \textup{d}t \\
	&\hspace{0.4cm} -\int_{0}^{\tau}  \int_{\overline{\Omega}} \nabla_x \widetilde{\textbf{u}} : \textup{d} \mathfrak{R} \  \textup{d}t.
	\end{aligned}
	\end{equation}
	
	\textbf{Step 3.} The class of functions $[\widetilde{\varrho}, \widetilde{\textbf{u}}]$ satisfying the relative energy inequality \eqref{relative energy inequality} can be extended by a density argument as long as integrals \eqref{e1}--\eqref{e8} remain well-defined. After a careful analysis, we recover the regularity class given by \eqref{rc 1}--\eqref{rc 3}. Notice that we have used the Sobolev embedding
	\begin{equation} \label{Sobolev embedding}
		W^{1, \frac{2\gamma}{\gamma+1}}\hookrightarrow L^{\gamma^*}(\Omega) \quad \mbox{with} \quad \gamma^*:= \frac{2\gamma d}{(d-2)\gamma+d}
	\end{equation}
	implying in particular that
	\begin{equation} \label{p}
		\varrho \in C_{\rm weak,loc}([0,\infty); L^p(\Omega)), \quad p:= \max \left\{ \gamma, \gamma^* \right\}
	\end{equation}
	and the fact that $\gamma^*> \gamma$ as long as $\{d=2\}$ or $\{d=3, \ d/2 < \gamma <d\}$ to get the optimal regularity for the density $\varrho$ and the momentum $\varrho \textbf{u}$.\\[0.1cm]
	
	\textbf{Step 4.} Let us now suppose that the couple $[\widetilde{\varrho}, \widetilde{\textbf{u}}]$ is a strong solution of problem \eqref{nsk continuity equation}--\eqref{nsk balance of momentum}, meaning that
	\begin{align*}
	\partial_t \widetilde{\varrho} + \divv_x (\widetilde{\varrho} \widetilde{\textup{\textbf{u}}})&=0, \\
	\partial_t \widetilde{\textup{\textbf{u}}} + \nabla_x \widetilde{\textup{\textbf{u}}} \cdot \widetilde{\textup{\textbf{u}}}+ \nabla_x P'(\widetilde{\varrho}) &= \frac{1}{\widetilde{\varrho}} \divv_x \big[ \mathbb{S}(\nabla_x \widetilde{\textup{\textbf{u}}}) + \mathbb{K}(\widetilde{\varrho}, \nabla_x\widetilde{\textbf{v}})\big], \\
	\partial_t \widetilde{\textbf{v}} + \nabla_x \widetilde{\textbf{v}} \cdot \widetilde{\textbf{u}} &= - \frac{1}{2\widetilde{\varrho}} \divv_x \big( \widetilde{\varrho} \nabla_x^{\top} \widetilde{\textbf{u}} \big),
	\end{align*}
	where the last one was deduced taking the gradient in the continuity equation \eqref{nsk continuity equation}. Then, the relative energy inequality \eqref{relative energy inequality} reduces to
	\begin{equation} \label{re1}
	\begin{aligned}
	&\left[ \int_{\Omega} E(\varrho, \textup{\textbf{u}}, \textbf{v} \ | \ \widetilde{\varrho}, \widetilde{\textup{\textbf{u}}}, \widetilde{\textbf{v}}) (t, \cdot) \  \textup{d}x\right]_{t=0}^{t=\tau} + \frac{1}{\lambda}\int_{\overline{\Omega}} \textup{d} \trace[\mathfrak{R}](\tau)+ \int_{0}^{\tau} \int_{\Omega} \mathbb{S}\big(\nabla_x (\textup{\textbf{u}}- \widetilde{\textbf{u}})\big):\nabla_x (\textup{\textbf{u}}-\widetilde{\textup{\textbf{u}}}) \ \textup{d}x \textup{d}t \\
	&\hspace{1.5cm} \leq - \int_{0}^{\tau} \int_{\Omega} \varrho \left[(\textup{\textbf{u}}-\widetilde{\textup{\textbf{u}}}) \otimes ( \textup{\textbf{u}}-\widetilde{\textup{\textbf{u}}}) + \hbar (\textbf{v}-\widetilde{\textbf{v}}) \otimes (\textbf{v} -\widetilde{\textbf{v}}) \right] :  \nabla_x \widetilde{\textup{\textbf{u}}} \ \textup{d}x \textup{d}t \\
	&\hspace{1.5cm} -\int_{0}^{\tau} \int_{\Omega}  [p(\varrho)- p'(\widetilde{\varrho})(\varrho-\widetilde{\varrho})-p(\widetilde{\varrho})] \divv_x \widetilde{\textup{\textbf{u}}} \ \textup{d}x \textup{d}t, \\
	&\hspace{1.5cm} - \int_{0}^{\tau}\int_{\Omega} \left( \frac{\varrho}{\widetilde{\varrho}} -1\right) (\textbf{u} -\widetilde{\textbf{u}}) \cdot \divv_x \mathbb{S}(\nabla_x \widetilde{\textbf{u}}) \ \textup{d}x \textup{d}t -\int_{0}^{\tau}  \int_{\overline{\Omega}} \nabla_x \widetilde{\textbf{u}} : \textup{d} \mathfrak{R} \  \textup{d}t.
	\end{aligned}
	\end{equation}
	On the one hand, we have
	\begin{equation*}
	\int_{0}^{\tau} \int_{\Omega} \mathbb{S}\big(\nabla_x (\textup{\textbf{u}}- \widetilde{\textbf{u}})\big):\nabla_x (\textup{\textbf{u}}-\widetilde{\textup{\textbf{u}}}) \ \textup{d}x \textup{d}t \geq \mu \int_{0}^{\tau} \int_{\Omega} |\nabla_x (\textbf{u}- \widetilde{\textbf{u}}) |^2 \ \textup{d}x \textup{d}t;
	\end{equation*}
	on the other hand, we have
	\begin{align*}
	| \varrho (\textup{\textbf{u}}-\widetilde{\textup{\textbf{u}}}) \otimes ( \textup{\textbf{u}}-\widetilde{\textup{\textbf{u}}})| &\leq c_1\frac{1}{2}  \trace [\varrho (\textup{\textbf{u}}-\widetilde{\textup{\textbf{u}}}) \otimes ( \textup{\textbf{u}}-\widetilde{\textup{\textbf{u}}})] = c_1\frac{1}{2} \varrho| \textbf{u}-\widetilde{\textbf{u}}|^2, \\[0.1cm]
	\hbar|\varrho (\textup{\textbf{v}}-\widetilde{\textup{\textbf{v}}}) \otimes ( \textup{\textbf{v}}-\widetilde{\textup{\textbf{v}}})| &\leq c_1\frac{\hbar}{2}  \trace [\varrho (\textup{\textbf{v}}-\widetilde{\textup{\textbf{v}}}) \otimes ( \textup{\textbf{v}}-\widetilde{\textup{\textbf{v}}})] = c_1\frac{\hbar}{2} \varrho| \textbf{v}-\widetilde{\textbf{v}}|^2, \\[0.1cm]
	p(\varrho)- p'(\widetilde{\varrho})(\varrho-\widetilde{\varrho})-p(\widetilde{\varrho}) &= (\gamma-1) \big[P(\varrho)- P'(\widetilde{\varrho})(\varrho-\widetilde{\varrho})-P(\widetilde{\varrho})\big], \\[0.1cm]
	|\mathfrak{R}| &\leq \frac{c_2}{\lambda} \trace[\mathfrak{R}].
	\end{align*}
	Moreover, it is easy to see that
	\begin{equation*}
	P(\varrho)-P'(\widetilde{\varrho}) (\varrho- \widetilde{\varrho}) -P(\widetilde{\varrho}) \geq c(\widetilde{\varrho}) \begin{cases}
	(\varrho-\widetilde{\varrho})^2 &\mbox{if } \varrho \in \left[\frac{\widetilde{\varrho}}{2}, 2\widetilde{\varrho}\right] \\
	(1+\varrho^{\gamma}) &\mbox{otherwise},
	\end{cases}
	\end{equation*}
	and therefore, it is possible to show that
	\begin{equation*}
	\int_{0}^{\tau}\int_{\Omega} \left| \frac{\varrho}{\widetilde{\varrho}} -1\right| \  |\textbf{u} -\widetilde{\textbf{u}}| \ \textup{d}x \textup{d}t \leq c(\delta) 	\int_{0}^{\tau}\int_{\Omega} E(\varrho, \textup{\textbf{u}} \ | \ \widetilde{\varrho}, \widetilde{\textup{\textbf{u}}}) \  \textup{d}x \textup{d}t + \delta  \int_{0}^{\tau} \int_{\Omega} |\textbf{u}-\widetilde{\textbf{u}}|^2 \ \textup{d}x \textup{d}t
	\end{equation*}
	for any $\delta >0$ (see for instance \cite{FeiJinNov}, Section 4.1.1). Therefore, from the Poincar\'{e} inequality and hypothesis \eqref{nsk same initial conditions}, we can rewrite \eqref{re1} as
	\begin{equation*}
	\begin{aligned}
	&\int_{\Omega} E(\varrho, \textup{\textbf{u}}, \textbf{v} \ | \ \widetilde{\varrho}, \widetilde{\textup{\textbf{u}}}, \widetilde{\textbf{v}}) (\tau, \cdot) \  \textup{d}x + \frac{1}{\lambda}\int_{\overline{\Omega}} \textup{d} \trace[\mathfrak{R}](\tau)+(1-\delta) \int_{0}^{\tau} \int_{\Omega} |\textbf{u}-\widetilde{\textbf{u}}|^2 \ \textup{d}x \textup{d}t \\
	&\hspace{1.5cm} \leq c(\delta, \widetilde{\varrho},\nabla_x \widetilde{\textbf{u}}, \divv_x \mathbb{S}( \nabla_x \widetilde{\textbf{u}})) \int_{0}^{\tau} \left( \int_{\Omega} E(\varrho, \textup{\textbf{u}}, \textbf{v} \ | \ \widetilde{\varrho}, \widetilde{\textup{\textbf{u}}}, \widetilde{\textbf{v}}) (t, \cdot) \  \textup{d}x + \frac{1}{\lambda}\int_{\overline{\Omega}} \textup{d} \trace[\mathfrak{R}](t) \right) \textup{d}t.
	\end{aligned}
	\end{equation*}
	Applying Gronwall's lemma, we can recover that
	\begin{equation*}
	\int_{\Omega} E(\varrho, \textup{\textbf{u}}, \textbf{v} \ | \ \widetilde{\varrho}, \widetilde{\textup{\textbf{u}}}, \widetilde{\textbf{v}}) (\tau, \cdot) \  \textup{d}x + \frac{1}{\lambda}\int_{\overline{\Omega}} \textup{d} \trace[\mathfrak{R}](\tau) \leq 0,
	\end{equation*}
	but since the quantity on the left-hand side is non-negative, this is possible if and only if \eqref{equivalent condition} holds.
	
	\subsection{Proof of Theorem \ref{ek weak-strong uniqueness}}
	
	We repeat the same passages performed before. Notice that in this case the relative energy functional is
	\begin{equation*}
	E (\varrho, \textbf{J}, \textbf{v} \ | \ \widetilde{\varrho}, \widetilde{\textbf{\textup{u}}}, \widetilde{\textbf{v}})=\frac{1}{2 \varrho}  |\textbf{J}- \varrho\widetilde{\textbf{\textup{u}}}|^2  + P(\varrho) -P'(\widetilde{\varrho}) (\varrho-\widetilde{\varrho}) -P(\widetilde{\varrho}) + \frac{\hbar}{2} \varrho \left|\textbf{v}-\widetilde{\textbf{v}}\right|^2,
	\end{equation*}
	and therefore the relative energy associated to the Euelr-Korteweg system \eqref{ek continuity equation}--\eqref{ek balance of momentum} reads
	\begin{align*}
	&\left[ \int_{\Omega} E(\varrho, \textup{\textbf{J}}, \textbf{v} \ | \ \widetilde{\varrho}, \widetilde{\textup{\textbf{u}}}, \widetilde{\textbf{v}}) (t, \cdot) \  \textup{d}x\right]_{t=0}^{t=\tau} + \frac{1}{\lambda}\int_{\overline{\Omega}} \textup{d} \trace[\mathfrak{R}](\tau) \\
	&\hspace{1.1cm}\leq -\int_{0}^{\tau} \int_{\Omega}  (\textup{\textbf{J}}-\varrho\widetilde{\textup{\textbf{u}}})  \cdot \left[\partial_t \widetilde{\textup{\textbf{u}}} + \nabla_x \widetilde{\textup{\textbf{u}}} \cdot \widetilde{\textup{\textbf{u}}}+ \nabla_x P'(\widetilde{\varrho}) - \frac{1}{\widetilde{\varrho}} \divv_x \mathbb{K}(\widetilde{\varrho}, \nabla_x\widetilde{\textbf{v}}) \right]  \textup{d}x \textup{d}t \\
	&\hspace{1.5cm}- \int_{0}^{\tau} \int_{\Omega}\frac{1}{\varrho} \big[(\textup{\textbf{J}}-\varrho\widetilde{\textup{\textbf{u}}}) \otimes (\textup{\textbf{J}}-\varrho\widetilde{\textup{\textbf{u}}})\big] : \nabla_x \widetilde{\textup{\textbf{u}}} \ \textup{d}x \textup{d}t \\
	&\hspace{1.5cm} -\int_{0}^{\tau} \int_{\Omega}  [p(\varrho)- p'(\widetilde{\varrho})(\varrho-\widetilde{\varrho})-p(\widetilde{\varrho})] \divv_x \widetilde{\textup{\textbf{u}}} \ \textup{d}x \textup{d}t \\
	&\hspace{1.5cm} -\int_{0}^{\tau} \int_{\Omega} p'(\widetilde{\varrho}) \left(  \frac{\varrho}{\widetilde{\varrho}}- 1\right) [\partial_t \widetilde{\varrho} + \divv_x (\widetilde{\varrho} \widetilde{\textup{\textbf{u}}})] \  \textup{d}x \textup{d}t \\
	&\hspace{1.5cm} -\hbar \int_{0}^{\tau}\int_{\Omega} \varrho (\textbf{v}-\widetilde{\textbf{v}}) \cdot \left[ \partial_t \widetilde{\textbf{v}} + \nabla_x \widetilde{\textbf{v}} \cdot \widetilde{\textbf{u}} + \frac{1}{\widetilde{\varrho}} \divv_x \big( \widetilde{\varrho} \nabla_x^{\top} \widetilde{\textbf{u}} \big) \right]  \textup{d}x \textup{d}t \\
	&\hspace{1.5cm} -\hbar \int_{0}^{\tau}\int_{\Omega} \varrho \big[(\textup{\textbf{v}}-\widetilde{\textup{\textbf{v}}}) \otimes (\textup{\textbf{v}}-\widetilde{\textup{\textbf{v}}})\big] : \nabla_x \widetilde{\textup{\textbf{u}}} \  \textup{d}x \textup{d}t \\
	&\hspace{1.5cm} -\int_{0}^{\tau}  \int_{\overline{\Omega}} \nabla_x \widetilde{\textbf{u}} : \textup{d} \mathfrak{R} \  \textup{d}t.
	\end{align*}
	If we suppose that the couple $[\widetilde{\varrho}, \widetilde{\textbf{u}}]$ is a strong solution then the previous inequality simplifies as
	\begin{equation*}
	\begin{aligned}
	&\left[ \int_{\Omega} E(\varrho, \textup{\textbf{J}}, \textbf{v} \ | \ \widetilde{\varrho}, \widetilde{\textup{\textbf{u}}}, \widetilde{\textbf{v}}) (t, \cdot) \  \textup{d}x\right]_{t=0}^{t=\tau} + \frac{1}{\lambda}\int_{\overline{\Omega}} \textup{d} \trace[\mathfrak{R}](\tau) \\
	&\hspace{1.5cm}- \int_{0}^{\tau} \int_{\Omega}\frac{1}{\varrho} \big[(\textup{\textbf{J}}-\varrho\widetilde{\textup{\textbf{u}}}) \otimes (\textup{\textbf{J}}-\varrho\widetilde{\textup{\textbf{u}}})\big] : \nabla_x \widetilde{\textup{\textbf{u}}} \ \textup{d}x \textup{d}t \\
	&\hspace{1.5cm} -\int_{0}^{\tau} \int_{\Omega}  [p(\varrho)- p'(\widetilde{\varrho})(\varrho-\widetilde{\varrho})-p(\widetilde{\varrho})] \divv_x \widetilde{\textup{\textbf{u}}} \ \textup{d}x \textup{d}t, \\
	&\hspace{1.5cm} -\hbar \int_{0}^{\tau}\int_{\Omega} \varrho \big[(\textup{\textbf{v}}-\widetilde{\textup{\textbf{v}}}) \otimes (\textup{\textbf{v}}-\widetilde{\textup{\textbf{v}}})\big] : \nabla_x \widetilde{\textup{\textbf{u}}} \  \textup{d}x \textup{d}t \\
	&\hspace{1.5cm} -\int_{0}^{\tau}  \int_{\overline{\Omega}} \nabla_x \widetilde{\textbf{u}} : \textup{d} \mathfrak{R} \  \textup{d}t,
	\end{aligned}
	\end{equation*}
	and therefore it is enough to proceed as before.
	
		\section{Existence}
	
	In this section, we aim to prove the existence of dissipative solutions for both the quantum Navier--Stokes and quantum Euler systems. More precisely, we will focus on the following two results.
	
	\begin{theorem}[Existence of dissipative solutions for the quatum Navier--Stokes system] \label{nsk existence}
		For any arbitrarily large $T>0$ and any fixed initial data
		\begin{equation*}
		[\varrho_{0},  \textbf{\textup{J}}_0, E_0] \in L^{\gamma}(\Omega) \times L^{\frac{2\gamma}{\gamma+1}} (\Omega; \mathbb{R}^d)\times [0,\infty)
		\end{equation*}
		such that
		\begin{equation*}
		\varrho_{0}>0, \quad \int_{\Omega}	\left[\frac{1}{2} \frac{|\textbf{\textup{J}}_0|^2}{\varrho_{0}}+ P(\varrho_{0})+\frac{\hbar}{2}\ |\nabla_x \sqrt{\varrho_{0}}|^2\right]\textup{d}x \leq E_0,
		\end{equation*}
		the quantum Navier--Stokes system \eqref{nsk continuity equation}--\eqref{nsk balance of momentum} with constitutive relations \eqref{pressure}--\eqref{viscosity} and boundary conditions \eqref{nsk boundary condition} admits a dissipative solution in the sense of Definition \ref{nsk dissipative solution}.
	\end{theorem}
	
	\begin{theorem}[Existence of dissipative solutions for the quantum Euler system] \label{ek existence}
		For any arbitrarily large $T>0$ and any fixed initial data
		\begin{equation*}
		[\varrho_{0},  \textbf{\textup{J}}_0, E_0] \in L^{\gamma}(\Omega) \times L^{\frac{2\gamma}{\gamma+1}} (\Omega; \mathbb{R}^d)\times [0,\infty)
		\end{equation*}
		such that
		\begin{equation*}
		\varrho_{0}>0, \quad \int_{\Omega}	\left[\frac{1}{2} \frac{|\textbf{\textup{J}}_0|^2}{\varrho_{0}}+ P(\varrho_{0})+\frac{\hbar}{2}\ |\nabla_x \sqrt{\varrho_{0}}|^2\right]\textup{d}x \leq E_0,
		\end{equation*}
		the  quantum Euler system \eqref{ek continuity equation}--\eqref{ek balance of momentum} with the isentropic pressure \eqref{pressure} and boundary conditions \eqref{ek boundary conditions} admits a dissipative solution in the sense of Definition \ref{ek dissipative solution}.
	\end{theorem}	
	The first theorem will be proved by employing a two-level approximation scheme based on addition of artificial viscosity terms, in order to convert the hyperbolic system into a parabolic one, and approximation via the Faedo-Galerkin technique. The second theorem will be obtained by letting the viscosity to go to zero in the quantum Navier--Stokes equations.

	\subsection{Proof of Theorem \ref{nsk existence}}
	
	From now one, let the time $T>0$ be fixed arbitrarily large. We  start by choosing a family $\{X_n\}_{n\in \mathbb{N}}$ of finite-dimensional spaces $X_n \subset L^2(\Omega; \mathbb{R}^d)$, such that
	\begin{equation*}
	X_n:= { \textup{span}} \{ \textbf{w}_i | \ \textbf{w}_i \in C_c^{\infty}(\Omega; \mathbb{R}^d), \ i=1, \dots, n \},
	\end{equation*}
	where $\textbf{w}_i$ are orthonormal with respect to the standard scalar product in $L^2(\Omega; \mathbb{R}^d)$. Now, for each $\varepsilon>0$ and $n \in \mathbb{N}$ fixed, we consider the following system
	\begin{align}
		\partial_t \varrho_{\varepsilon, n} +\divv_x (\varrho_{\varepsilon, n} \textbf{u}_{\varepsilon, n}) &= \varepsilon \Delta_x \varrho_{\varepsilon, n}, \label{approximation continuity equation} \\
			\partial_t(\varrho_{\varepsilon, n} \textbf{u}_{\varepsilon, n}) + \divv_x(\varrho_{\varepsilon, n} \textbf{u}_{\varepsilon, n}\otimes \textbf{u}_{\varepsilon, n}) &+\nabla_x p(\varrho_{\varepsilon, n}) +\varepsilon \nabla_x \textbf{u}_{\varepsilon, n} \cdot \nabla_x \varrho_{\varepsilon, n}  \label{approximation momentum equation}\\
			&= \divv_x \big( \mathbb{S}(\nabla_x \textbf{u}_{\varepsilon, n}) + \mathbb{K}(\varrho_{\varepsilon, n}, \nabla_x \varrho_{\varepsilon, n}, \nabla_x^2 \varrho_{\varepsilon, n})  \big) \notag
	\end{align}
	on $(0,T) \times \Omega$, where we look for approximated velocities
	\begin{equation*}
		\textbf{u}_{\varepsilon,n} \in C([0,T]; X_n).
	\end{equation*}
	Moreover, we impose the homogeneous Neumann and no-slip boundary conditions for the density and velocity, respectively
	\begin{equation} \label{boundary conditions approximated system}
		\nabla_x \varrho_{\varepsilon,n} \cdot \textbf{n} |_{\partial \Omega} =0, \quad \textbf{u}_{\varepsilon,n}|_{\partial \Omega}=0,
	\end{equation}
	and we fix the initial conditions
	\begin{equation*}
		\varrho_{\varepsilon,n}(0,\cdot)= \varrho_{0,n}, \quad 	(\varrho_{\varepsilon,n} \textbf{u}_{\varepsilon,n})(0,\cdot)= \textbf{J}_0 \quad \mbox{on }\Omega,
	\end{equation*}
	where the initial densities $\{\varrho_{0,n}\}_{n \in \mathbb{N}}\subset W^{1,2}(\Omega)$, $0< \underline{\varrho}_{n} \leq \varrho_{0,n} \leq \overline{ \varrho}_n < \infty$, are chosen in such a way that
	\begin{equation*}
		\varrho_{0,n} \rightarrow \varrho_0 \ \mbox{in } L^1(\Omega) \ \mbox{as } n\rightarrow \infty,
	\end{equation*}
	with the couple $(\varrho_0, \textbf{J}_0)$ as in the hypotheses of Theorem \ref{nsk existence}. Solvability of the approximated problem will be discussed in the following sections.
	
	\subsubsection{On the approximated continuity equation}
	
	For any fixed $\varepsilon>0$, $n \in \mathbb{N}$ and given $\textbf{u}_{\varepsilon,n} \in C([0,T]; X_n)$, let us focus on finding that unique weak  solution $\varrho_{\varepsilon,n} = \varrho[\textbf{u}_{\varepsilon,n}]$ of equation \eqref{approximation continuity equation}. Before stating the existence result for the approximated continuity equation, notice that since $X_n$ is finite-dimensional, all the norms on $X_n$ induced by $W^{k,p}$-norms, with $k\in \mathbb{N}$ and $1\leq p\leq \infty$, are equivalent; thus, we deduce that
	\begin{equation*}
		\textbf{u}_{\varepsilon,n} \in L^{\infty}(0,T; W^{1,\infty}(\Omega; \mathbb{R}^d)),
	\end{equation*}
	and there exist two constants $0<\underline{n}<\overline{n}<\infty$, depending solely on the dimension $n$ of $X_n$, such that for any $t\in [0,T]$
	\begin{equation} \label{connection norm in Xn and W}
		\underline{n} \| \textbf{u}_{\varepsilon,n}(t, \cdot) \|_{W^{1,\infty}(\Omega)} \leq \| \textbf{u}_{\varepsilon,n}(t, \cdot) \|_{X_n} \leq \overline{n} \| \textbf{u}_{\varepsilon,n}(t, \cdot) \|_{W^{1,\infty}(\Omega)}.
	\end{equation}
	
	\begin{lemma} \label{existence approximated densities}
		Let $\Omega \subset \mathbb{R}^d$ be a bounded domain of class $C^2$ and let $\varepsilon>0$, $n \in \mathbb{N}$ be fixed. For any given $\textup{\textbf{u}}_{\varepsilon,n} \in C([0,T]; X_n)$, there exists a unique solution
		\begin{equation*}
			\varrho_{\varepsilon, n} \in L^2((0,T); W^{2,2}(\Omega)) \cap W^{1,2}(0,T;L^2(\Omega))
		\end{equation*}
		of equation \eqref{approximation continuity equation} with $\varrho_{\varepsilon, n}(0,\cdot)=\varrho_{0,n}$. Moreover,
		\begin{itemize}
			\item[(i)] (bound from above - maximum principle) the weak solution $\varrho_{\varepsilon, n}$ satisfies
			\begin{equation} \label{bound above density}
			\| \varrho_{\varepsilon, n} \|_{L^{\infty}((0,\tau) \times \Omega)} \leq \overline{\varrho}_n \exp \left( \tau \|\divv_x \textup{\textbf{u}}_{\varepsilon, n}\|_{L^{\infty}((0,T) \times \Omega)} \right),
			\end{equation}
			for any $\tau \in [0,T]$, with
			\begin{equation} \label{maximum approximated initial density}
			\overline{\varrho}_n:= \max_{\Omega}  \varrho_{0,n};
			\end{equation}
			\item[(ii)] (bound from below) the weak solution $\varrho_{\varepsilon, n}$ satisfies
			\begin{equation} \label{bound below density}
			\essinf_{(0,\tau) \times \Omega} \varrho_{\varepsilon, n}(t,x)\geq \underline{\varrho}_n \exp \left( -\tau \|\divv_x \textup{\textbf{u}}_{\varepsilon, n}\|_{L^{\infty}((0,T) \times \Omega)}\right),
			\end{equation}
			for any $\tau \in [0,T]$, with
			\begin{equation} \label{minimum approximated initial density}
			\underline{\varrho}_n:= \min_{\Omega}  \varrho_{0,n};
			\end{equation}
			\item[(iii)] let $\textup{\textbf{u}}_1, \textup{\textbf{u}}_2 \in C([0,T]; X_n)$ be such that
			\begin{equation*}
			\max_{i=1,2} \| \textup{\textbf{u}}_i\|_{L^{\infty}(0,T; W^{1,\infty}(\Omega;\mathbb{R}^d))} \leq K
			\end{equation*}
			with $K \in (0,\infty)$, and let $\varrho_i= \varrho[\textup{\textbf{u}}_i]$, $i=1,2$ be the weak solutions of the approximated continuity equation \eqref{approximation continuity equation}sharing the same initial data $\varrho_{0,n}$. Then, for any $\tau \in [0,T]$,
			\begin{equation} \label{difference two solution approximates densities}
			\| (\varrho_1- \varrho_2)(\tau, \cdot)\|_{L^2(\Omega)} \leq c_1 \| \textup{\textbf{u}}_1 -\textup{\textbf{u}}_2 \|_{L^{\infty}(0,\tau; W^{1,\infty}(\Omega; \mathbb{R}^d))}
			\end{equation}
			with $c_1=c_1(\varepsilon, \varrho_0, T, K)$.
		\end{itemize}
	\end{lemma}
	\begin{proof}
		The proof is a straightforward consequence of Lemma 4.3 in \cite{ChaJinNov}.
	\end{proof}
	
	\subsubsection{On the approximated balance of momentum}
	
	Let us now turn our attention to the approximated balance of momentum \eqref{approximation momentum equation}. The approximate velocities $\textbf{u}_{\varepsilon, n} \in C([0,T]; X_n)$ are looked for to satisfy the integral identity
	\begin{equation} \label{projection momentum equation to finite-dimesional space}
	\begin{aligned}
	&\left[ \int_{\Omega}\varrho_{\varepsilon, n} \textbf{u}_{\varepsilon, n}(t,\cdot) \cdot \bm{\psi}\ \textup{d}x\right]_{t=0}^{t=\tau} = \int_{0}^{\tau} \int_{\Omega} \left[ (\varrho_{\varepsilon, n} \textbf{u}_{\varepsilon, n} \otimes \textbf{u}_{\varepsilon, n}) : \nabla_x \bm{\psi} +p(\varrho_{\varepsilon, n})\divv_x \bm{\psi}\right] \textup{d}x\textup{d}t \\
	&\hspace{2cm}+ \frac{\hbar}{4}  \int_{0}^{\tau}\int_{\Omega} \big[ \nabla_x \varrho_{\varepsilon, n}\cdot  \divv_x \nabla_x^{\top}\bm{\psi} +4 (\nabla_x \sqrt{\varrho_{\varepsilon, n}} \otimes \nabla_x \sqrt{\varrho_{\varepsilon, n}}): \nabla_x \bm{\psi} \big] \ \textup{d}x \textup{d}t \\
	&\hspace{2cm}-\int_{0}^{\tau}\int_{\Omega}\mathbb{S} (\nabla_x \textup{\textbf{u}}_{\varepsilon, n}): \nabla_x \bm{\psi} \ \textup{d}x\textup{d}t - \varepsilon \int_{0}^{\tau}\int_{\Omega} \nabla_x \varrho_{\varepsilon, n} \cdot \nabla_x \textbf{u}_{\varepsilon, n} \cdot \bm{\psi} \ \textup{d}x \textup{d}t
	\end{aligned}
	\end{equation}
	for any test function $\bm{\psi} \in X_n$ and all $\tau \in [0,T]$, with $(\varrho_{\varepsilon,n} \textbf{u}_{\varepsilon,n})(0,\cdot)= \textbf{J}_0$. Now, the integral identity \eqref{projection momentum equation to finite-dimesional space} can be rephrased for any $\tau\in [0,T]$ as
	\begin{equation*}
	\langle \mathscr{M}[\varrho_{\varepsilon, n}(\tau, \cdot)] (\textbf{u}_{\varepsilon, n}(\tau, \cdot)), \bm{\psi}\rangle = \langle \textbf{J}_0^*, \bm{\psi} \rangle + \langle \int_{0}^{\tau} \mathscr{N} [\varrho_{\varepsilon, n}(s, \cdot), \textbf{u}_{\varepsilon, n}(s,\cdot)] \ \textup{d}s, \bm{\psi} \rangle
	\end{equation*}
	with
	\begin{align*}
		\mathscr{M}[\varrho]: X_n \rightarrow X_n^*, \quad &\langle \mathscr{M}[\varrho] \textbf{v}, \textbf{w} \rangle :=\int_{\Omega}\varrho \textbf{v} \cdot \textbf{w} \ \textup{d}x, \\
		\textbf{J}_0^* \in X_n^*, \quad &	\langle \textbf{J}_0^*, \bm{\psi} \rangle:= \int_{\Omega} \textbf{J}_0 \cdot \bm{\psi}  \ \textup{d}x, \\
		\mathscr{N}[\varrho, \textbf{u}] \in X_n^*, \quad & \langle  \mathscr{N} [\varrho, \textbf{u}], \bm{\psi}\rangle :=\int_{\Omega} \left[ (\varrho \textbf{u} \otimes \textbf{u}) : \nabla_x \bm{\psi} +p(\varrho)\divv_x \bm{\psi}\right] \textup{d}x \\
		&\hspace{2cm} + \frac{\hbar}{4} \int_{\Omega} \big[ \nabla_x \varrho\cdot  \divv_x \nabla_x^{\top}\bm{\psi} +4 (\nabla_x \sqrt{\varrho} \otimes \nabla_x \sqrt{\varrho}): \nabla_x \bm{\psi} \big] \ \textup{d}x \\
		&\hspace{2cm} - \int_{\Omega}\mathbb{S} (\nabla_x \textup{\textbf{u}}): \nabla_x \bm{\psi} \ \textup{d}x - \varepsilon \int_{\Omega} \nabla_x \varrho \cdot \nabla_x \textbf{u} \cdot \bm{\psi} \ \textup{d}x.
	\end{align*}
	We are now ready to apply the following lemma.
	
	\begin{lemma}
		Let
		\begin{equation*}
			\mathcal{B}(0, \underline{n}K):= \left\{ \textbf{\textup{v}} \in C([0,T(n)]; X_n)\big| \sup_{t\in[0,T(n)]} \| \textbf{\textup{v}}(t, \cdot)\|_{X_n}\leq \underline{n}K \right\},
		\end{equation*}
		with $\underline{n}$ defined as in \eqref{connection norm in Xn and W}. For $K>0$ sufficiently large and $T(n)$ sufficiently small, the map
		\begin{equation*}
			\mathscr{F}: \mathcal{B}(0, \underline{n}K) \rightarrow C([0,T(n)]; X_n)
		\end{equation*}
		such that
		\begin{equation*}
			\mathscr{F}[\textbf{\textup{u}}_{\varepsilon, n}](\tau, \cdot):= \mathscr{M}^{-1} [\varrho_{\varepsilon, n}(\tau, \cdot)] \left( \textbf{\textup{J}}_0^*+ \int_{0}^{\tau} \mathscr{N} [\varrho_{\varepsilon, n}(s, \cdot), \textbf{\textup{u}}_{\varepsilon, n}(s,\cdot)] \ \textup{d}s \right),
		\end{equation*}
		 is a contraction mapping from the closed ball $\mathcal{B}(0, \underline{n}K)$ onto itself and therefore it admits a unique fixed point $\textbf{\textup{u}}_{\varepsilon, n} \in C([0,T(n)]; X_n)$.
	\end{lemma}

	\begin{proof}
		The lemma is a straightforward consequence of the Banach-Cacciopoli fixed point theorem. Notice in particular that $\varrho_{\varepsilon, n}=\varrho[\textbf{u}_{\varepsilon, n}]$ is the weak solution of equation \eqref{approximation continuity equation} uniquely determined by $\textbf{u}_{\varepsilon, n}$ and thus by Lemma \ref{existence approximated densities} we can deduce that $0<\underline{\varrho}_n e^{-Kt} \leq \varrho_{\varepsilon, n}(t,x) \leq \overline{\varrho}_n e^{Kt}$ for any $t\in [0,T(n)]$ whenever $\textbf{u}_{\varepsilon, n} \in \mathcal{B}(0, \underline{n}K)$, where $\overline{\varrho}_n$, $\underline{\varrho}_n$ are defined as in \eqref{maximum approximated initial density}, \eqref{minimum approximated initial density} respectively. Therefore, the operator $\mathscr{M}$ is invertible and from \eqref{difference two solution approximates densities} it is also easy to show that $\mathscr{F}$ is a contraction mapping, see e.g. \cite{ChaJinNov}, Section 4.3.2 for more details.
	\end{proof}
	
	So far, we have found the velocity $\textbf{u}_{\varepsilon, n}$ solving the integral identity \eqref{projection momentum equation to finite-dimesional space} on the time interval $[0,T(n)]$. However, the previous procedure can be repeated a finite number of times until we reach $T=T(n)$, as long as we have a bound on $\textbf{u}_{\varepsilon, n}$ independent of $T(n)$; in other words, we need some \textit{energy estimates}. We have that
	\begin{equation} \label{projection velocities}
	\begin{aligned}
		\int_{\Omega}\partial_t(\varrho_{\varepsilon, n} \textbf{u}_{\varepsilon, n}) \cdot \bm{\psi}\ \textup{d}x &= \int_{\Omega} \left[ (\varrho_{\varepsilon, n} \textbf{u}_{\varepsilon, n} \otimes \textbf{u}_{\varepsilon, n}) : \nabla_x \bm{\psi} +p(\varrho_{\varepsilon, n})\divv_x \bm{\psi}\right] \textup{d}x \\
		&- \int_{\Omega}\big[ \mathbb{K}(\varrho_{\varepsilon, n}, \nabla_x \varrho_{\varepsilon, n}, \nabla_x^2 \varrho_{\varepsilon, n}) + \mathbb{S} (\nabla_x \textup{\textbf{u}}_{\varepsilon, n}) \big] : \nabla_x\bm{\psi}  \ \textup{d}x \\
		&- \varepsilon \int_{\Omega} \nabla_x \varrho_{\varepsilon, n} \cdot \nabla_x \textbf{u}_{\varepsilon, n} \cdot \bm{\psi} \ \textup{d}x
	\end{aligned}
	\end{equation}
	holds on $(0,T(n))$ for any $\bm{\psi} \in X_n$, with $\varrho_{\varepsilon, n}= \varrho[\textbf{u}_{\varepsilon, n}]$. We can then take $\bm{\psi}= \textbf{u}_{\varepsilon, n}(t,\cdot)$
	\begin{align*}
		\frac{\textup{d}}{\textup{d}t}\int_{\Omega} \frac{1}{2} \varrho_{\varepsilon, n} |\textbf{u}_{\varepsilon, n}|^2  \ \textup{d}x =&- \int_{\Omega} \mathbb{S} (\nabla_x \textup{\textbf{u}}_{\varepsilon, n}): \nabla_x \textup{\textbf{u}}_{\varepsilon, n}  \ \textup{d}x + \int_{\Omega} p(\varrho_{\varepsilon, n})\divv_x \textup{\textbf{u}}_{\varepsilon, n} \ \textup{d}x \\
		&- \int_{\Omega} \mathbb{K}(\varrho_{\varepsilon, n}, \nabla_x \varrho_{\varepsilon, n}, \nabla_x^2 \varrho_{\varepsilon, n}) : \nabla_x \textup{\textbf{u}}_{\varepsilon, n} \ \textup{d}x \\
		& - \int_{\Omega} \frac{1}{2} |\textup{\textbf{u}}_{\varepsilon, n}|^2 \big( \partial_t \varrho_{\varepsilon, n} + \divv_x (\varrho_{\varepsilon, n} \textup{\textbf{u}}_{\varepsilon, n}) - \varepsilon \Delta_x \varrho_{\varepsilon, n} \big) \textup{d}x,
	\end{align*}
	where the last line vanishes due to \eqref{approximation continuity equation}. Multiplying \eqref{approximation continuity equation} by $P'(\varrho)$ we recover that in this context the pressure potential $P=P(\varrho)$ satisfies the following identity
	\begin{equation*}
		p(\varrho_{\varepsilon, n}) \divv_x \textbf{u}_{\varepsilon, n} = - \partial_t P(\varrho_{\varepsilon, n}) -\divv_x(P(\varrho_{\varepsilon, n}) \textbf{u}_{\varepsilon, n}) +\varepsilon P'(\varrho_{\varepsilon, n}) \Delta_x \varrho_{\varepsilon, n} .
	\end{equation*}
	Therefore, the previous integral identity can be rewritten as
	\begin{equation} \label{part appr energy}
		\begin{aligned}
		\frac{\textup{d}}{\textup{d}t}\int_{\Omega} \left(\frac{1}{2} \varrho_{\varepsilon, n} |\textbf{u}_{\varepsilon, n}|^2 + P(\varrho_{\varepsilon, n})\right)  \ \textup{d}x =&- \int_{\Omega} \mathbb{S} (\nabla_x \textup{\textbf{u}}_{\varepsilon, n}): \nabla_x \textup{\textbf{u}}_{\varepsilon, n}  \ \textup{d}x - \varepsilon \int_{\Omega} P''(\varrho_{\varepsilon, n})|\nabla_x \varrho_{\varepsilon, n}|^2 \ \textup{d}x \\
		&- \int_{\Omega} \mathbb{K}(\varrho_{\varepsilon, n}, \nabla_x \varrho_{\varepsilon, n}, \nabla_x^2 \varrho_{\varepsilon, n}) : \nabla_x \textup{\textbf{u}}_{\varepsilon, n} \ \textup{d}x.
		\end{aligned}
	\end{equation}
	Moreover, we have
	\begin{align*}
	- \int_{\Omega} \mathbb{K}&(\varrho_{\varepsilon, n}, \nabla_x \varrho_{\varepsilon, n}, \nabla_x^2 \varrho_{\varepsilon, n}) : \nabla_x \textup{\textbf{u}}_{\varepsilon, n} \ \textup{d}x = \int_{\Omega} \divv_x \big[  \mathbb{K}(\varrho_{\varepsilon, n}, \nabla_x \varrho_{\varepsilon, n}, \nabla_x^2 \varrho_{\varepsilon, n}) \big] \cdot \textup{\textbf{u}}_{\varepsilon, n} \ \textup{d}x \\
	&= \frac{\hbar}{2}\int_{\Omega} \varrho_{\varepsilon, n} \nabla_x \left( \frac{\Delta_x \sqrt{\varrho_{\varepsilon, n}}}{\sqrt{\varrho_{\varepsilon, n}}} \right) \cdot \textbf{u}_{\varepsilon, n}  \ \textup{d}x = \frac{\hbar}{2} \int_{\Omega}  \frac{\Delta_x \sqrt{\varrho_{\varepsilon, n}}}{\sqrt{\varrho_{\varepsilon, n}}} \  \divv_x(\varrho_{\varepsilon, n} \textbf{u}_{\varepsilon, n}) \ \textup{d}x \\
	&= \frac{\hbar}{4} \int_{\Omega} \left( \frac{1}{\varrho_{\varepsilon, n}} \Delta_x \varrho_{\varepsilon, n} -\frac{1}{2 \varrho_{\varepsilon, n}^2}| \nabla_x \varrho_{\varepsilon, n}|^2 \right) \  (\partial_t \varrho_{\varepsilon, n} -\varepsilon \Delta_x \varrho_{\varepsilon, n}) \ \textup{d}x \\
	&= - \frac{\textup{d}}{\textup{d}t} \int_{\Omega} \frac{\hbar}{2}  |\nabla_x \sqrt{\varrho_{\varepsilon, n}}|^2 \ \textup{d}x -  \varepsilon \frac{\hbar}{4} \int_{\Omega} \varrho_{\varepsilon, n} \  |\nabla_x^2(\log \varrho_{\varepsilon, n})|^2\ \textup{d}x,
	\end{align*}
	where we used formulas
	\begin{align*}
		\Delta_x f(\varrho) &= f'(\varrho) \Delta_x \varrho + f''(\varrho) |\nabla_x \varrho|^2, \\
		|\nabla_x^2 f(\varrho)|^2 &= \frac{1}{2} \Delta_x |\nabla_x f(\varrho)|^2 - \nabla_x f(\varrho) \cdot \nabla_x \Delta_x f(\varrho),
	\end{align*}
	to write
	\begin{align*}
	\int_{\Omega} \left( \frac{1}{\varrho} \Delta_x \varrho -\frac{1}{2 \varrho^2}| \nabla_x \varrho|^2 \right) \Delta_x \varrho \ \textup{d}x &= \int_{\Omega} \Delta_x (\log \varrho) \  \Delta_x \varrho \ \textup{d}x +  \frac{1}{2} \int_{\Omega} \left| \nabla_x  (\log \varrho) \right|^2 \Delta_x \varrho \ \textup{d}x \\
	&=  \int_{\Omega} \varrho \left(-\nabla_x \Delta_x (\log \varrho) \cdot \nabla_x(\log \varrho) + \frac{1}{2} \Delta_x|\nabla_x (\log \varrho)|^2\right) \ \textup{d}x \\
	&=  \int_{\Omega} \varrho \  |\nabla_x^2 (\log \varrho)|^2 \ \textup{d}x.
	\end{align*}
	We have finally obtained
	\begin{align*}
		\frac{\textup{d}}{\textup{d}t}\int_{\Omega} &\left(\frac{1}{2} \varrho_{\varepsilon, n} |\textbf{u}_{\varepsilon, n}|^2 + P(\varrho_{\varepsilon, n}) + \frac{\hbar}{2} |\nabla_x \sqrt{\varrho_{\varepsilon, n}}|^2\right)  \ \textup{d}x \\
		&+ \int_{\Omega} \mathbb{S} (\nabla_x \textup{\textbf{u}}_{\varepsilon, n}): \nabla_x \textup{\textbf{u}}_{\varepsilon, n}  \ \textup{d}x \\
		&+ \varepsilon \int_{\Omega} \left(P''(\varrho_{\varepsilon, n})|\varrho_{\varepsilon, n} \textbf{v}_{\varepsilon, n}|^2 + \frac{\hbar}{4} \ \varrho_{\varepsilon, n}|\nabla_x^2 (\log \varrho_{\varepsilon,n})|^2\right) \ \textup{d}x=0.
	\end{align*}
	Integrating the previous expression over $(0,\tau)$, we get the following energy inequality
	\begin{equation} \label{second approximate energy equality}
	\begin{aligned}
	\int_{\Omega} &\left[ \frac{1}{2} \varrho_{\varepsilon, n} |\textbf{u}_{\varepsilon, n}|^2 +P(\varrho_{\varepsilon, n}) + \frac{\hbar}{2}\ |\nabla_x \sqrt{\varrho_{\varepsilon, n}}|^2\right](\tau, \cdot) \ \textup{d}x \\
	&+ \int_{0}^{\tau}\int_{\Omega} \mathbb{S}(\nabla_x \textbf{u}_{\varepsilon, n}): \nabla_x \textbf{u}_{\varepsilon, n} \ \textup{d}x \textup{d}t  \\
	&+ \varepsilon \int_{0}^{\tau} \int_{\Omega} \left(P''(\varrho_{\varepsilon, n}) |\nabla_x \varrho_{\varepsilon, n}|^2 + \frac{\hbar}{4}\varrho_{\varepsilon, n}  \  |\nabla_x^2 (\log \varrho_{\varepsilon, n}) |^2 \right) \textup{d}x \textup{d}t \\
	&\leq \int_{\Omega} \left[\frac{1}{2} \frac{|\textbf{J}_0|^2}{\varrho_{0,n}}+ P(\varrho_{0,n})+\frac{\hbar}{2}\ |\nabla_x \sqrt{\varrho_{0,n}}|^2\right]\textup{d}x,
	\end{aligned}
	\end{equation}
	for any time $\tau \in [0,T(n)]$. In particular, if we suppose that
	\begin{equation} \label{initial energies independent of n}
	\int_{\Omega} \left[\frac{1}{2} \frac{|\textbf{J}_0|^2}{\varrho_{0,n}} +P(\varrho_{0,n}) + \frac{\hbar}{2}\ |\nabla_x \sqrt{\varrho_{0,n}}|^2\right] \textup{d}x \leq E_0
	\end{equation}
	where the constant $E_0$ is independent of $n>0$, the term on the left-hand side of \eqref{second approximate energy equality} is bounded. Consequently, it is not difficult to show that the functions $\textbf{u}_{\varepsilon, n}(t, \cdot)$ remain bounded in $X_n$ for any $t$ independently of $T(n)\leq T$. Thus we are allowed to iterate the previous local existence result to construct a solution defined on the whole time interval $[0,T]$, see e.g. the last part of Section 7.3.4 in \cite{Fei1} for more details.
	
	Summarizing, so far we proved the following result.
	
	\begin{lemma} \label{existence in epsilon n}
		For every fixed $\varepsilon>0$, $n\in \mathbb{N}$, and any $\varrho_{0,n} \in W^{1,2}(\Omega)$ such that
		\begin{equation*}
		\int_{\Omega} \left[\frac{1}{2} \frac{|\textup{\textbf{J}}_0|^2}{\varrho_{0,n}} +P(\varrho_{0,n}) + \frac{\hbar}{2}\ |\nabla_x \sqrt{\varrho_{0,n}}|^2\right] \textup{d}x \leq E_0,
		\end{equation*}
		where the constant $E_0$ is independent of $n$, there exist
		\begin{align*}
		\varrho_{\varepsilon, n} &\in L^2((0,T); W^{2,2}(\Omega)) \cap W^{1,2}(0,T;L^2(\Omega)), \\
		\textup{\textbf{u}}_{ \varepsilon, n} &\in C([0,T]; X_n),
		\end{align*}
		such that
		\begin{itemize}
			\item[(i)] the integral identity
			\begin{equation} \label{weak formulation continuity equation epsilon}
			\left[\int_{\Omega} \varrho_{\varepsilon, n} \varphi (t, \cdot) \ \textup{d}x\right]_{t=0}^{t=\tau} = \int_{0}^{\tau} \int_{\Omega} (\varrho_{\varepsilon, n} \partial_t \varphi +\varrho_{\varepsilon, n} \textup{\textbf{u}}_{\varepsilon, n} \cdot \nabla_x \varphi -\varepsilon \nabla_x \varrho_{\varepsilon, n} \cdot \nabla_x \varphi ) \ \textup{d} x \textup{d}t
			\end{equation}
			holds for any $\tau \in [0,T]$ and any $\varphi \in C^1([0,T]\times \overline{\Omega})$, with $\varrho_{\varepsilon, n}(0, \cdot)=\varrho_{0,n}$;
			\item[(ii)] the integral identity
			\begin{equation} \label{weak formulation momentum equation epsilon}
			\begin{aligned}
			&\left[\int_{\Omega}\varrho_{\varepsilon, n} \textup{\textbf{u}}_{\varepsilon, n} \cdot \bm{\varphi}(t, \cdot)\ \textup{d}x\right]_{t=0}^{t=\tau}  \\
			&\hspace{2cm}= \int_{0}^{\tau}\int_{\Omega} \left[ \varrho_{\varepsilon, n}\textup{\textbf{u}}_{\varepsilon, n} \cdot \partial_t\bm{\varphi}+(\varrho_{\varepsilon, n} \textup{\textbf{u}}_{\varepsilon, n} \otimes \textup{\textbf{u}}_{\varepsilon, n}) : \nabla_x \bm{\varphi} +p(\varrho_{\varepsilon, n})\divv_x \bm{\varphi} \right] \textup{d}x\textup{d}t \\
			&\hspace{2cm}+ \frac{\hbar}{4}  \int_{0}^{\tau}\int_{\Omega} \big[ \nabla_x \varrho_{\varepsilon, n} \cdot \divv_x \nabla_x^{\top}\bm{\varphi} + 4 (\nabla_x \sqrt{\varrho_{\varepsilon, n}} \otimes \nabla_x \sqrt{\varrho_{\varepsilon, n}}): \nabla_x \bm{\varphi} \big] \textup{d}x \textup{d}t  \\
			&\hspace{2cm}-\int_{0}^{\tau}\int_{\Omega}\mathbb{S} (\nabla_x \textup{\textbf{u}}_{\varepsilon, n}): \nabla_x \bm{\varphi} \ \textup{d}x\textup{d}t -\varepsilon \int_{0}^{\tau}\int_{\Omega} \nabla_x \varrho_{\varepsilon, n}\cdot \nabla_x \textup{\textbf{u}}_{\varepsilon, n} \cdot \bm{\varphi} \ \textup{d}x\textup{d}t
			\end{aligned}
			\end{equation}
			holds for any $\tau \in [0,T]$ and any $\bm{\varphi} \in C^1([0,T]; X_n)$, with $(\varrho_{\varepsilon, n} \textup{\textbf{u}}_{\varepsilon, n})(0, \cdot)= \textup{\textbf{J}}_0$;
			\item[(iii)] the integral inequality
			\begin{equation} \label{energy inequality in epsilon and n}
			\begin{aligned}
				\int_{\Omega} &\left[ \frac{1}{2} \varrho_{\varepsilon, n} |\textbf{\textup{u}}_{\varepsilon, n}|^2 +P(\varrho_{\varepsilon, n}) + \frac{\hbar}{2}\ |\nabla_x \sqrt{\varrho_{\varepsilon, n}}|^2\right](\tau, \cdot) \ \textup{d}x \\
				&+ \int_{0}^{\tau}\int_{\Omega} \mathbb{S}(\nabla_x \textbf{\textup{u}}_{\varepsilon, n}): \nabla_x \textbf{\textup{u}}_{\varepsilon, n} \ \textup{d}x \textup{d}t  \\
				&+ \varepsilon \int_{0}^{\tau} \int_{\Omega} \left(P''(\varrho_{\varepsilon, n}) |\nabla_x \varrho_{\varepsilon, n}|^2 + \frac{\hbar}{4}\varrho_{\varepsilon, n}  \  |\nabla_x^2 (\log \varrho_{\varepsilon, n}) |^2 \right) \textup{d}x \textup{d}t \\
				&\leq \int_{\Omega} \left[\frac{1}{2} \frac{|\textbf{\textup{J}}_0|^2}{\varrho_{0,n}}+ P(\varrho_{0,n})+\frac{\hbar}{2}\ |\nabla_x \sqrt{\varrho_{0,n}}|^2\right]\textup{d}x,
			\end{aligned}
			\end{equation}
			holds for any time $\tau \in [0,T]$.
		\end{itemize}
	\end{lemma}
	
	\subsubsection{Limit $\varepsilon \rightarrow 0$}
	
	In order to perform the limit $\varepsilon \rightarrow 0$, we need the following result.
	
	\begin{lemma} \label{convergences epsilon }
		Let $n\in \mathbb{N}$ be fixed and let $\{ \varrho_{\varepsilon,n}, \textup{\textbf{u}}_{\varepsilon,n} \}_{\varepsilon>0}$ be as in Lemma \ref{existence in epsilon n}. Then, passing to suitable subsequences as the case may be, the following convergences hold as $\varepsilon \rightarrow 0$.
		\begin{align}
		\varrho_{\varepsilon,n} \overset{*}{\rightharpoonup} \varrho_n \quad &\mbox{in } L^{\infty}((0,T)\times \Omega) \label{convergence densities epsilon}\\[0.1cm]
		\textup{\textbf{u}}_{\varepsilon,n} \overset{*}{\rightharpoonup} \textup{\textbf{u}}_n \quad &\mbox{in } L^{\infty}(0,T; W^{1,\infty}(\Omega; \mathbb{R}^d)), \label{convergence velocities epsilon} \\[0.1cm]
		\varrho_{\varepsilon,n}\textup{\textbf{u}}_{\varepsilon,n} \overset{*}{\rightharpoonup} \varrho_n \textup{\textbf{u}}_n\quad &\mbox{in } L^{\infty}((0,T)\times \Omega; \mathbb{R}^d), \label{convergence momenta epsilon} \\[0.1cm]
		\varrho_{\varepsilon,n}\textup{\textbf{u}}_{\varepsilon,n} \otimes \textup{\textbf{u}}_{\varepsilon,n} \overset{*}{\rightharpoonup} \varrho_n \textup{\textbf{u}}_n \otimes \textup{\textbf{u}}_n \quad &\mbox{in } L^{\infty}((0,T)\times \Omega; \mathbb{R}^{d\times d}), \label{convergence convective terms epsilon}\\[0.1cm]
		\nabla_x \varrho_{\varepsilon,n} \overset{*}{\rightharpoonup} \nabla_x \varrho_n \quad &\mbox{in } L^{\infty}(0,T;L^2(\Omega)), \label{convergence gradient density}\\[0.1cm]
		p(\varrho_{\varepsilon,n}) \overset{*}{\rightharpoonup} \overline{p(\varrho_n)} \quad &\mbox{in } L^{\infty}(0,T; \mathcal{M}(\overline{\Omega})), \label{convergence pressure}\\[0.1cm]
		\nabla_x \sqrt{\varrho_{\varepsilon,n}} \otimes 	\nabla_x \sqrt{\varrho_{\varepsilon,n}} \overset{*}{\rightharpoonup} \overline{\nabla_x \sqrt{\varrho_n} \otimes \nabla_x \sqrt{\varrho_n}} \quad &\mbox{in } L^{\infty}(0,T; \mathcal{M}(\overline{\Omega}; \mathbb{R}^{d\times d})), \label{convergence tensor product sqrt density}\\[0.1cm]
		\sqrt{\varepsilon}\  \nabla_x \varrho_{\varepsilon,n}\rightharpoonup \sqrt{\varepsilon}\  \nabla_x \varrho_n \quad &\mbox{in }L^2((0,T) \times \Omega; \mathbb{R}^d), \label{convergence f epsilon}\\[0.1cm]
		\sqrt{\varepsilon}\  \nabla_x \varrho_{\varepsilon,n} \cdot \nabla_x  \textup{\textbf{u}}_{\varepsilon,n} \rightharpoonup \sqrt{\varepsilon} \  \overline{ \nabla_x \varrho_n \cdot \nabla_x  \textup{\textbf{u}}}_n \quad &\mbox{in } L^2((0,T) \times \Omega; \mathbb{R}^d). \label{convergence g epsilon}
		\end{align}
	\end{lemma}
	
	\begin{proof}
		From \eqref{energy inequality in epsilon and n} it is easy to deduce the following uniform bounds
		\begin{align}
		\| P(\varrho_{\varepsilon,n}) \|_{L^{\infty}(0,T; L^1(\Omega))} &\leq c(\overline{E}) \label{unform bound pressure potential}\\
		\| \nabla_x \sqrt{\varrho_{\varepsilon,n}}\|_{L^{\infty}(0,T; L^2(\Omega; \mathbb{R}^d) )} &\leq c(\overline{E}) \label{uniform bound gradient sqr density} \\
		\|  \nabla_x \textbf{u}_{\varepsilon,n}\|_{L^2((0,T) \times \Omega; \mathbb{R}^{d\times d})} &\leq c(\overline{E}) \label{uniform bound velocities}.
		\end{align}
		Estimate \eqref{uniform bound velocities} combined with the Poincar\'{e} inequality provides
		\begin{equation*}
		\|\textbf{u}_{\varepsilon,n}\|_{L^2(0,T; W^{1,2}(\Omega; \mathbb{R}^d))}\leq c_1
		\end{equation*}
		for some positive constant $c_1$ independent of $\varepsilon>0$. Applying Lemma \ref{existence approximated densities}, we get
		\begin{equation} \label{bound densities}
		e^{-c_1T} \underline{\varrho}_n \leq \varrho_{\varepsilon,n}(t,x) \leq e^{c_1T} \overline{\varrho}_n, \quad \mbox{for all }(t,x) \in [0,T] \times \overline{\Omega}.
		\end{equation}
		which yields to convergence \eqref{convergence densities epsilon}. From the fact that $\textbf{u}_{\varepsilon,n}$ belongs to $C([0,T]; X_n)$, it is easy to deduce
		\begin{equation} \label{bound velocities}
		\sup_{t\in [0,T]} \| \textbf{u}_{\varepsilon,n}(t,\cdot)\|_{W^{1,\infty}(\Omega; \mathbb{R}^d)} \leq c_2,
		\end{equation}
		from which convergence \eqref{convergence velocities epsilon} follows. Combining \eqref{bound densities} and \eqref{bound velocities}, we can recover
		\begin{equation*}
		\varrho_{\varepsilon,n}\textbf{u}_{\varepsilon,n} \overset{*}{\rightharpoonup} \overline{\varrho_n \textbf{u}_n} \quad \mbox{in } L^{\infty}((0,T)\times \Omega; \mathbb{R}^d).
		\end{equation*}
		Now, notice that \eqref{convergence densities epsilon} can be strengthened to
		\begin{equation*}
		\varrho_{\varepsilon,n} \rightarrow \varrho_n \quad \mbox{in } C_{\textup{weak}} ([0,T]; L^p(\Omega)) \quad \mbox{for all } 1<p<\infty
		\end{equation*}
		as $\varepsilon \rightarrow 0$, so that, relaying on the compact Sobolev embedding
		\begin{equation*}
		L^p(\Omega) \hookrightarrow \hookrightarrow W^{-1,1}(\Omega) \quad \mbox{for all } p\geq 1,
		\end{equation*}
		we obtain
		\begin{equation*}
		\varrho_{\varepsilon,n} \rightarrow \varrho_n \quad \mbox{in } C([0,T]; W^{-1,1}(\Omega))
		\end{equation*}
		as $\varepsilon \rightarrow 0$. The last convergence combined with \eqref{convergence velocities epsilon}, implies
		\begin{equation*}
			\overline{\varrho_n \textbf{u}_n}= \varrho_n \textbf{u}_n  \quad \mbox{a.e. in }(0,T) \times \Omega,
		\end{equation*}
		and thus, we get \eqref{convergence momenta epsilon}. Similarly, from \eqref{convergence velocities epsilon} and \eqref{convergence momenta epsilon} we can deduce \eqref{convergence convective terms epsilon}. Noticing that
		\begin{equation} \label{gradient rho}
		\nabla_x \varrho_{\varepsilon,n} = 2 \sqrt{\varrho_{\varepsilon,n}} \ \nabla_x \sqrt{\varrho_{\varepsilon,n}},
		\end{equation}
		convergence \eqref{convergence gradient density} can be deduced combining \eqref{uniform bound gradient sqr density} and \eqref{bound densities}. From \eqref{unform bound pressure potential} and \eqref{uniform bound gradient sqr density}, we can deduce that the sequences $\{ p(\varrho_{\varepsilon,n}) \}_{\varepsilon >0}$, $\{ \nabla_x \sqrt{\varrho_{\varepsilon,n}} \otimes \nabla_x \sqrt{\varrho_{\varepsilon,n}} \}_{\varepsilon>0}$ are uniformly bounded in $L^{\infty}(0,T; L^1(\Omega))$. However, since the $L^1$-space cannot be identified as the dual space of any separable space and therefore it is not possible to apply the Banach-Alaoglu theorem, a suitable idea consists in the embedding of $L^1(\Omega)$ in the space of measures $\mathcal{M}(\overline{\Omega})$, which, on the contrary, is the dual space of $C(\overline{\Omega})$; we get convergences \eqref{convergence pressure}, \eqref{convergence tensor product sqrt density}. Finally, from \eqref{bound densities} and the energy inequality \eqref{energy inequality in epsilon and n}, we have
		\begin{equation*}
		\varepsilon \int_{0}^{\tau} \int_{\Omega} |\nabla_x \varrho_{\varepsilon,n}|^2 \  \textup{d}x \textup{d}t \leq  \varepsilon \ c(\underline{\varrho}_n) \int_{0}^{\tau} \int_{\Omega} P''(\varrho_{\varepsilon,n}) |\nabla_x \varrho_{\varepsilon,n}|^2 \ \textup{d}x \textup{d}t \leq c(\underline{\varrho}_n, T).
		\end{equation*}
		In this way we get \eqref{convergence f epsilon} and, in view of \eqref{bound velocities}, \eqref{convergence g epsilon}.
	\end{proof}
	
	We are now ready to let $\varepsilon \rightarrow 0$ in the weak formulations \eqref{weak formulation continuity equation epsilon}, \eqref{weak formulation momentum equation epsilon}; notice in particular that, in view of \eqref{convergence f epsilon} and \eqref{convergence g epsilon}, for any $\tau \in [0,T]$, any $\varphi \in C^1([0,T] \times \overline{\Omega})$ and any $\bm{\varphi} \in C^1([0,T]; X_n)$
	\begin{align*}
	\varepsilon \int_{0}^{\tau} \int_{\Omega} \nabla_x \varrho_{\varepsilon,n} \cdot \nabla_x \varphi \ \textup{d}x\textup{d}t &= \sqrt{\varepsilon} \int_{0}^{\tau} \int_{\Omega} \sqrt{\varepsilon} \ \nabla_x \varrho_{\varepsilon,n} \cdot \nabla_x \varphi \ \textup{d}x\textup{d}t \rightarrow 0, \\
	\varepsilon \int_{0}^{\tau}\int_{\Omega} \nabla_x \varrho_{\varepsilon,n}\cdot \nabla_x \textbf{u}_{\varepsilon,n} \cdot \bm{\varphi} \ \textup{d}x\textup{d}t &= \sqrt{\varepsilon} \int_{0}^{\tau}\int_{\Omega} \sqrt{\varepsilon} \ \nabla_x \varrho_{\varepsilon,n} \cdot \nabla_x  \textup{\textbf{u}}_{\varepsilon,n}\cdot \bm{\varphi} \ \textup{d}x\textup{d}t \rightarrow 0
	\end{align*}
	as $\varepsilon \rightarrow 0$. We therefore obtain that the weak formulations of the continuity equation \eqref{nsk weak formulation continuity equation} and balance of momentum \eqref{nsk balance of momentum} hold for any $\tau \in [0,T]$ and any $\varphi \in C^1([0,T]\times \overline{\Omega})$, $\bm{\varphi} \in C^1([0,T]; X_n)$, respectively, with the Reynolds stresses
	\begin{equation*}
	\mathfrak{R}_n \in L^{\infty} (0,T; \mathcal{M}(\overline{\Omega}; \mathbb{R}_{\textup{sym}}^{d\times d}))
	\end{equation*}
	such that
	\begin{equation*}
	\textup{d}\mathfrak{R}_n: = (\overline{p(\varrho_n)}- p(\varrho_n)) \mathbb{I} \ \textup{d}x + \hbar \ (\overline{\nabla_x \sqrt{\varrho_n} \otimes \nabla_x \sqrt{\varrho_n}} - \nabla_x \sqrt{\varrho_n} \otimes \nabla_x \sqrt{\varrho_n}) \ \textup{d}x.
	\end{equation*}
	We claim that $\mathfrak{R}_n$ are positive measure, meaning that
	\begin{equation} \label{positivity defect}
	\mathfrak{R}_n \in L^{\infty} (0,T; \mathcal{M}^+(\overline{\Omega}; \mathbb{R}_{\textup{sym}}^{d\times d}));
	\end{equation}
	more precisely, we have to show that for any $\bm{\xi} \in \mathbb{R}^d$ and any bounded open set $\mathcal{B} \subset \Omega$
	\begin{equation*}
	\mathfrak{R}_n: (\bm{\xi} \otimes \bm{\xi})\geq 0 \quad \mbox{in } \mathcal{D}' ((0,T) \times \mathcal{B}).
	\end{equation*}
 	We have
	\begin{equation*}
	\mathfrak{R}_n: (\bm{\xi} \otimes \bm{\xi}) = (\overline{p(\varrho_n)}- p(\varrho_n)) |\bm{\xi}|^2 + \hbar \ (\overline{\nabla_x \sqrt{\varrho_n} \otimes \nabla_x \sqrt{\varrho_n}} - \nabla_x \sqrt{\varrho_n} \otimes \nabla_x \sqrt{\varrho_n}): (\bm{\xi} \otimes \bm{\xi});
	\end{equation*}
	on the one hand, the first term on the right-hand side is non-negative due to the convexity of the function $\varrho \mapsto p(\varrho)$ and therefore $p(\varrho_n) \leq \overline{p(\varrho_n)}$ (see e.g. \cite{Fei}, Theorem 2.1.1); on the other hand, we have
	\begin{align*}
	&(\overline{\nabla_x \sqrt{\varrho_n} \otimes \nabla_x \sqrt{\varrho_n}} - \nabla_x \sqrt{\varrho_n} \otimes \nabla_x \sqrt{\varrho_n}): (\bm{\xi} \otimes \bm{\xi}) \\[0.1cm]
	& \hspace{1cm}= \lim_{\varepsilon \rightarrow 0} \left[ (\nabla_x \sqrt{\varrho_{\varepsilon,n}} \otimes \nabla_x \sqrt{\varrho_{\varepsilon,n}} ): (\bm{\xi} \otimes \bm{\xi}) \right] - (\nabla_x \sqrt{\varrho_n} \otimes \nabla_x \sqrt{\varrho_n}): (\bm{\xi} \otimes \bm{\xi}) \\[0.1cm]
	&\hspace{1cm}= \lim_{\varepsilon \rightarrow 0} |\nabla_x \sqrt{\varrho_{\varepsilon,n}} \cdot \bm{\xi}|^2 - |\nabla_x \sqrt{\varrho_n} \cdot \bm{\xi}|^2 = \overline{|\nabla_x \sqrt{\varrho_n} \cdot \bm{\xi}|^2} - |\nabla_x \sqrt{\varrho_n} \cdot \bm{\xi}|^2
	\end{align*}
	in $\mathcal{D}' ((0,T) \times \mathcal{B})$, where it is interesting to notice that the derivatives of the function $\varrho \mapsto f_{\bm{\xi}}(\varrho)= |\nabla_x \sqrt{\varrho} \cdot \bm{\xi}|^2$ are such that
	\begin{equation*}
	f^{(k)}_{\bm{\xi}} (\varrho) = (-1)^k \ \frac{k!}{\varrho^k} \  f_{\bm{\xi}} (\varrho) \quad \mbox{for any } k \in \mathbb{N};
	\end{equation*}
	in particular, $f_{\bm{\xi}}$ is convex for any fixed $\bm{\xi} \in \mathbb{R}^d$ and therefore $f_{\bm{\xi}}(\varrho_n) \leq \overline{f_{\bm{\xi}}(\varrho_n)} $. We get \eqref{positivity defect}.
	
	Similarly, we can pass to the limit in \eqref{energy inequality in epsilon and n} to get
	\begin{equation} \label{energy inequality with energy defect}
	\begin{aligned}
	&\int_{\Omega} \left[ \frac{1}{2} \varrho_n |\textbf{\textup{u}}_n|^2 +P(\varrho_n) + \frac{\hbar}{2}\ |\nabla_x \sqrt{\varrho_n}|^2\right](\tau, \cdot) \ \textup{d}x \\
	&\hspace{2cm}+ \int_{\overline{\Omega}} \textup{d} \mathfrak{E}_n(\tau) + \int_{0}^{\tau}\int_{\Omega} \mathbb{S}(\nabla_x \textbf{\textup{u}}_n): \nabla_x \textbf{\textup{u}}_n \ \textup{d}x \textup{d}t  \\
	&\hspace{2cm}\leq \int_{\Omega} \left[\frac{1}{2} \frac{|\textbf{\textup{J}}_0|^2}{\varrho_{0,n}}+ P(\varrho_{0,n})+\frac{\hbar}{2}\ |\nabla_x \sqrt{\varrho_{0,n}}|^2\right]\textup{d}x,
	\end{aligned}
	\end{equation}
	with
	\begin{equation*}
	\mathfrak{E}_n \in L^{\infty}(0,T; \mathcal{M}^+(\overline{\Omega}))
	\end{equation*}
	such that
	\begin{equation*}
	\textup{d}\mathfrak{E}_n = \left(\overline{P(\varrho_n)}- P(\varrho_n)\right) \textup{d}x + \frac{\hbar}{2} \left(\overline{|\nabla_x \sqrt{\varrho_n}|^2} -|\nabla_x \sqrt{\varrho_n}|^2\right) \textup{d}x.
	\end{equation*}
	Furthermore, introducing $\lambda=\lambda(d, \gamma)= \max \{ d(\gamma-1), 2 \}$, we obtain
	\begin{align*}
	\trace[\mathfrak{R}_n] &= d\left(\overline{p(\varrho_n)}- p(\varrho_n)\right) + \hbar \left( \lim_{\varepsilon \rightarrow 0} \trace[\nabla_x \sqrt{\varrho_{\varepsilon,n}} \otimes \nabla_x \sqrt{\varrho_{\varepsilon,n}} ] - \trace[\nabla_x \sqrt{\varrho_n} \otimes \nabla_x \sqrt{\varrho_n}] \right) \\
	&= d(\gamma-1) \left(\overline{P(\varrho_n)}- P(\varrho_n) \right) + \hbar \left(\overline{|\nabla_x \sqrt{\varrho_n}|^2} -|\nabla_x \sqrt{\varrho_n}|^2\right) \leq \lambda \mathfrak{E}_n
	\end{align*}
	and therefore, the energy inequality \eqref{energy inequality with energy defect} will still hold replacing $\mathfrak{E}_n$ with $\lambda^{-1} \trace[\mathfrak{R}_n]$.

	\begin{lemma} \label{limit epsilon to zero}
		For every fixed $n\in \mathbb{N}$, and any $\varrho_{0,n} \in C(\overline{\Omega})$ such that
		\begin{equation*}
		\int_{\Omega} \left[\frac{1}{2} \frac{|\textup{\textbf{J}}_0|^2}{\varrho_{0,n}} +P(\varrho_{0,n})+ \frac{\hbar}{2}\ |\nabla_x \sqrt{\varrho_{0,n}}|^2\right] \textup{d}x \leq E_0,
		\end{equation*}
		where the constant $E_0$ is independent of $n$, there exist
		\begin{align*}
		\varrho_n &\in L^{\infty}(0,T; L^{\infty}(\Omega) ), \\
		\textup{\textbf{u}}_n &\in C([0,T]; X_n),
		\end{align*}
		with
		\begin{equation*}
		\quad e^{-cT} \underline{\varrho}_n \leq \varrho_n(t,x) \leq e^{cT} \overline{\varrho}_n, \quad \mbox{for all }(t,x) \in [0,T] \times \overline{\Omega},
		\end{equation*}
		for a positive constant $c$, such that
		\begin{itemize}
			\item[(i)] the integral identity
			\begin{equation} \label{weak formulation continuity equation n}
			\left[\int_{\Omega} \varrho_n \varphi (t, \cdot) \ \textup{d}x\right]_{t=0}^{t=\tau} = \int_{0}^{\tau} \int_{\Omega} (\varrho_n \partial_t \varphi +\varrho_n \textup{\textbf{u}}_n\cdot \nabla_x \varphi ) \ \textup{d} x
			\end{equation}
			holds for any $\tau \in [0,T]$ and any $\varphi \in C^1([0,T]\times \overline{\Omega})$, with $\varrho_n(0, \cdot)=\varrho_{0,n}$;
			\item[(ii)] there exists
			\begin{equation*}
			\mathfrak{R}_n  \in L^{\infty} (0,T; \mathcal{M}^+(\overline{\Omega}; \mathbb{R}_{\textup{sym}}^{d\times d}))
			\end{equation*}
			such that the integral identity
			\begin{equation} \label{weak formulation balance of momentum n}
			\begin{aligned}
			\left[\int_{\Omega}\varrho_n \textup{\textbf{u}}_n \cdot \bm{\varphi}(t, \cdot)\ \textup{d}x\right]_{t=0}^{t=\tau}  &= \int_{0}^{\tau}\int_{\Omega} \left[ \varrho_n\textup{\textbf{u}}_n \cdot \partial_t\bm{\varphi}+(\varrho_n \textup{\textbf{u}}_n \otimes \textup{\textbf{u}}_n) : \nabla_x \bm{\varphi} +p(\varrho_n)\divv_x \bm{\varphi} \right] \textup{d}x\textup{d}t \\
			&+ \frac{\hbar}{4}  \int_{0}^{\tau}\int_{\Omega} \big[ \nabla_x \varrho_n \cdot  \divv_x \nabla_x^{\top} \bm{\varphi} + 4 (\nabla_x \sqrt{\varrho_n} \otimes \nabla_x \sqrt{\varrho_n}): \nabla_x \bm{\varphi} \big] \textup{d}x \textup{d}t  \\
			&-\int_{0}^{\tau}\int_{\Omega}\mathbb{S} (\nabla_x \textup{\textbf{u}}_n): \nabla_x \bm{\varphi} \ \textup{d}x\textup{d}t + \int_{0}^{\tau}  \int_{\overline{\Omega}} \nabla_x \bm{\varphi} : \textup{d} \mathfrak{R}_n \  \textup{d}t
			\end{aligned}
			\end{equation}
			holds for any $\tau \in [0,T]$ and any $\bm{\varphi} \in C^1([0,T]; X_n)$, with $(\varrho_n \textup{\textbf{u}}_n)(0, \cdot)= \textup{\textbf{J}}_0$;
			\item[(iii)] there exists a positive constant $\lambda= \lambda(d, \gamma)$ such that the integral inequality
			\begin{equation} \label{energy inequality in n}
			\begin{aligned}
			\begin{aligned}
			&\int_{\Omega} \left[ \frac{1}{2} \varrho_n |\textbf{\textup{u}}_n|^2 +P(\varrho_n) + \frac{\hbar}{2}\ |\nabla_x \sqrt{\varrho_n}|^2\right](\tau, \cdot) \ \textup{d}x \\
			&\hspace{2cm}+ \frac{1}{\lambda}\int_{\overline{\Omega}} \textup{d} \trace[\mathfrak{R}_n](\tau) + \int_{0}^{\tau}\int_{\Omega} \mathbb{S}(\nabla_x \textbf{\textup{u}}_n): \nabla_x \textbf{\textup{u}}_n \ \textup{d}x \textup{d}t  \\
			&\hspace{2cm}\leq \int_{\Omega} \left[\frac{1}{2} \frac{|\textbf{\textup{J}}_0|^2}{\varrho_{0,n}}+ P(\varrho_{0,n})+\frac{\hbar}{2}\ |\nabla_x \sqrt{\varrho_{0,n}}|^2\right]\textup{d}x,
			\end{aligned}
			\end{aligned}
			\end{equation}
			holds for a.e. $\tau \in (0,T)$.
		\end{itemize}
	\end{lemma}
	
	\subsubsection{Limit $n \rightarrow \infty$} \label{Limit n to infinity}
	
	In order to perform the last limit, we need the following result.
	
	\begin{lemma}
		Let $\{ \varrho_n, \textbf{\textup{u}}_n, \mathfrak{R}_n \}_{n \in \mathbb{N}}$ be as in Lemma \ref{limit epsilon to zero}. Then, passing to suitable subsequences as the case may be, the following convergences hold as $n \rightarrow \infty$.
		\begin{align}
		\varrho_n \rightarrow \varrho \quad &\mbox{in } C_{\textup{weak}}([0,T]; L^{\gamma}(\Omega)) \label{convergence densities n}\\[0.1cm]
		\varrho_n\textup{\textbf{u}}_n \rightarrow  \varrho \textup{\textbf{u}}\quad &\mbox{in } C_{\textup{weak}}([0,T]; L^q(\Omega; \mathbb{R}^d)), \label{convergence momenta n} \\[0.1cm]
		\textup{\textbf{u}}_n \rightharpoonup \textup{\textbf{u}} \quad &\mbox{in } L^2(0,T; W^{1,2}(\Omega; \mathbb{R}^d)), \label{convergence velocities n} \\[0.1cm]
		\varrho_n\textup{\textbf{u}}_n \otimes \textup{\textbf{u}}_n \rightharpoonup \varrho \textup{\textbf{u}} \otimes \textup{\textbf{u}} \quad &\mbox{in } L^2(0,T; L^r(\Omega; \mathbb{R}^{d\times d})), \ r>1, \label{convergence convective terms n}\\[0.1cm]
		\varrho_{n} \overset{*}{\rightharpoonup} \varrho \quad &\mbox{in } L^{\infty}(0,\infty; W^{1,\frac{2\gamma}{\gamma+1}}(\Omega) ), \label{convergence gradient density n}\\[0.1cm]
		p(\varrho_n) \overset{*}{\rightharpoonup} \overline{p(\varrho)} \quad &\mbox{in } L^{\infty}(0,T; \mathcal{M}(\overline{\Omega})), \label{convergence pressure n}\\[0.1cm]
		\nabla_x \sqrt{\varrho_n} \otimes 	\nabla_x \sqrt{\varrho_n} \overset{*}{\rightharpoonup} \overline{\nabla_x \sqrt{\varrho} \otimes \nabla_x \sqrt{\varrho}} \quad &\mbox{in } L^{\infty}(0,T; \mathcal{M}(\overline{\Omega})), \label{convergence tensor product sqrt density n}\\[0.1cm]
		\mathfrak{R}_n \overset{*}{\rightharpoonup} \widetilde{\mathfrak{R}} \quad &\mbox{in } L^{\infty}(0,T; \mathcal{M}(\overline{\Omega}; \mathbb{R}_{\textup{sym}}^{d\times d})) \label{convergence defects},
		\end{align}
		where the exponent $q$ is defined as in \eqref{weak convergence momenta}.
	\end{lemma}
	
	\begin{proof}
		From the energy inequality \eqref{energy inequality in n} we can recover the following uniform bounds:
		\begin{align}
		\| \sqrt{\varrho_n} \textbf{u} \|_{L^{\infty}(0,T; L^2(\Omega; \mathbb{R}^d))} & \leq c(\overline{E}), \label{uniform bound kinetic energy}\\
		\| \nabla_x \sqrt{\varrho_n}\|_{L^{\infty}(0,T; L^2(\Omega; \mathbb{R}^d) )} &\leq c(\overline{E}) \label{uniform bound gradient sqr density n} \\
		\| P(\varrho_n) \|_{L^{\infty}(0,T; L^1(\Omega))} &\leq c(\overline{E}), \label{uniform bound pressure potential n}\\
		\| \trace[\mathfrak{R}_n] \|_{L^{\infty}(0,T; L^1(\Omega))} &\leq c(\overline{E}), \label{uniform bounds defects}\\
		\|  \nabla_x \textbf{u}_n\|_{L^2((0,T) \times \Omega; \mathbb{R}^{d\times d})} &\leq c(\overline{E}) \label{uniform bound velocities n}.
		\end{align}
		From \eqref{uniform bound pressure potential n}, it is easy to deduce that, passing to a suitable subsequence,
		\begin{equation} \label{convergence densities in Linfinity}
		\varrho_n \overset{*}{\rightharpoonup} \varrho \quad \mbox{in } L^{\infty} (0,T; L^{\gamma}(\Omega));
		\end{equation}
		this convergence can be strengthened to \eqref{convergence densities n} as a consequence of the Arzel\`{a}-Ascoli theorem. Convergence \eqref{convergence densities n} combined with identity \eqref{gradient rho}, the uniform bound \eqref{uniform bound gradient sqr density n} and the fact that $\gamma>\frac{2\gamma}{\gamma+1}$ imply \eqref{convergence gradient density n}. Convergence \eqref{convergence velocities n} can be recovered from \eqref{uniform bound velocities n}, while from \eqref{uniform bound kinetic energy}, \eqref{convergence densities in Linfinity}, \eqref{convergence gradient density n}, the Sobolev embedding \eqref{Sobolev embedding} and the fact that for a.e. $t\in [0,T]$, as a consequence of H\"{o}lder inequality,
		\begin{equation*}
		\| (\varrho_n \textbf{u}_n) (t,\cdot) \|_{L^q(\Omega; \mathbb{R}^d)} \leq \| (\sqrt{\varrho_n} \textbf{u})(t,\cdot) \|_{L^2(\Omega; \mathbb{R}^d)} \| \sqrt{\varrho_n}(t,\cdot)\|_{L^{2p}(\Omega)},
		\end{equation*}
		with $q$ and $p$ defined as in \eqref{weak convergence momenta} and \eqref{p}, respectively, we get
		\begin{equation} \label{convergence momenta in Linfinity}
		\varrho_n \textbf{u}_n \overset{*}{\rightharpoonup} \overline{\varrho \textbf{u}} \quad \mbox{in } L^{\infty} (0,T; L^q(\Omega; \mathbb{R}^d)).
		\end{equation}
		Now, from the compact Sobolev embedding $L^{\gamma}(\Omega) \hookrightarrow\hookrightarrow W^{-1,2}(\Omega)$, we get the strong convergence of the densities in $C([0,T]; W^{-1,2}(\Omega))$ and therefore
		\begin{equation*}
		\overline{\varrho \textbf{u}}= \varrho \textbf{u} \quad \mbox{a.e. on } (0,T) \times \Omega.
		\end{equation*}
		Once again, convergence \eqref{convergence momenta in Linfinity} can be strengthened to \eqref{convergence momenta n}. Next, convergences \eqref{convergence momenta n}, \eqref{convergence velocities n} combined with the Sobolev embedding $L^q(\Omega)\hookrightarrow\hookrightarrow W^{-1,2}(\Omega)$ imply \eqref{convergence convective terms n}, where the exponent $r$ must satisfy
		\begin{equation*}
		\frac{1}{r} = \frac{1}{q} + \frac{d-2}{2d}.
		\end{equation*}
		Finally, convergences \eqref{convergence pressure n} and \eqref{convergence tensor product sqrt density n} can be deduced from \eqref{uniform bound gradient sqr density n} and \eqref{uniform bound pressure potential n} respectively, repeating the same passages performed in the proof of Lemma \ref{convergences epsilon }, while convergence \eqref{convergence defects} follows from \eqref{uniform bounds defects}.
	\end{proof}
	
	We are now ready to let $n \rightarrow \infty$.  Once again, we get that the weak formulations of the continuity equation \eqref{nsk weak formulation continuity equation} and balance of momentum \eqref{nsk balance of momentum} hold for any $\tau \in [0,T]$ and any $\varphi \in C^1([0,T]\times \overline{\Omega})$, $\bm{\varphi} \in C^1([0,T]; X_n)$, respectively, with the Reynolds stress
	\begin{equation*}
	\mathfrak{R} \in L^{\infty}(0,T; \mathcal{M}^+(\Omega; \mathbb{R}^{d\times d}_{\textup{sym}}))
	\end{equation*}
	such that
	\begin{equation*}
	\textup{d} \mathfrak{R} = \textup{d} \widetilde{\mathfrak{R}} \ + \left(\overline{p(\varrho)}- p(\varrho)\right) \mathbb{I} \ \textup{d}x + \hbar \left(\overline{\nabla_x \sqrt{\varrho} \otimes \nabla_x \sqrt{\varrho}} - \nabla_x \sqrt{\varrho} \otimes \nabla_x \sqrt{\varrho}\right) \textup{d}x.
	\end{equation*}
	Choosing a c\`{a}gl\`{a}d function $E=E(\tau)$ such that
	\begin{equation*}
		E(\tau) =\int_{\Omega} \left[ \frac{1}{2} \varrho |\textbf{\textup{u}}|^2 +P(\varrho) + \frac{\hbar}{2}\ |\nabla_x \sqrt{\varrho}|^2\right](\tau, \cdot) \ \textup{d}x + \frac{1}{\lambda}\int_{\overline{\Omega}} \textup{d} \trace[\mathfrak{R}](\tau)
	\end{equation*}
	for a.e. $\tau \in (0,T)$, the integral inequality
	\begin{equation*}
	\begin{aligned}
 		E(\tau)+ \int_{0}^{\tau}\int_{\Omega} \mathbb{S}(\nabla_x \textbf{\textup{u}}): \nabla_x \textbf{\textup{u}} \ \textup{d}x \textup{d}t \leq \int_{\Omega} \left[\frac{1}{2} \frac{|\textbf{\textup{J}}_0|^2}{\varrho_{0}}+ P(\varrho_{0})+\frac{\hbar}{2}\ |\nabla_x \sqrt{\varrho_{0}}|^2\right]\textup{d}x,
	\end{aligned}
	\end{equation*}
	holds for a.e. $\tau \in (0,T)$. Finally, notice that the spaces $X_n$ can be chosen in such a way that the validity of \eqref{nsk weak formulation balance of momentum} can be extended to any $\bm{\varphi} \in C^1([0,T]; C^2_c(\Omega; \mathbb{R}^d))$ by a density argument. Given
	\begin{equation*}
	\bm{\varphi} \in C^1([0,T]; C^2(\overline{\Omega}; \mathbb{R}^d)) , \ \bm{\varphi}|_{\partial \Omega}=0,
	\end{equation*}
	we can construct a sequence $\{ \bm{\varphi}_n \}_{n\in \mathbb{N}} \subset C^1([0,T]; C^2_c(\Omega; \mathbb{R}^d))$ such that
	\begin{equation*}
	\{\bm{\varphi}_n\}_{n \in \mathbb{N}} \mbox{ is uniformly bounded in } W^{1,\infty}(0,T; W^{2,\infty}(\Omega; \mathbb{R}^d))
	\end{equation*}
	and, for any $(t,x) \in (0,T) \times \Omega$
	\begin{align*}
	\bm{\varphi}_n(t,x) \rightarrow \bm{\varphi}(t,x), \quad &\partial_t \bm{\varphi}_n (t,x)\rightarrow \partial_t \bm{\varphi}(t,x), \\
	\nabla_x \bm{\varphi}_n(t,x) \rightarrow \nabla_x \bm{\varphi}(t,x), \quad &\nabla_x^2 \bm{\varphi}_n(t,x) \rightarrow \nabla_x^2 \bm{\varphi}(t,x).
	\end{align*}
	This concludes the proof of Theorem \ref{nsk existence}.
	
	\subsection{Proof of Theorem \ref{ek existence}} \label{Existence Euler}
	
	Let $\{ [\varrho_{\delta}, \textbf{u}_{\delta}] \}_{\delta >0}$ be a family of dissipative solutions of the quantum Navier--Stokes system
	\begin{align}
	\partial_t \varrho_{\delta} + \divv_x (\varrho_{\delta}\textbf{u}_{\delta})&=0, \label{nsk  delta continuity equation}\\
	\partial_t (\varrho_{\delta} \textbf{u}_{\delta}) + \divv_x \left( \varrho_{\delta}\textbf{u}_{\delta} \otimes \textbf{u}_{\delta} \right) + \nabla_x p(\varrho_{\delta}) &= \delta \divv_x \mathbb{S}(\nabla_x \textbf{u}_{\delta})+ \divv_x \mathbb{K}(\varrho_{\delta}, \nabla_x \varrho_{\delta}, \nabla_x^2 \varrho_{\delta}), \label{nsk delta balance of momentum}
	\end{align}
	with correspondent Reynolds stress $\mathfrak{R}_{\delta}$, pressure \eqref{pressure}, viscous stress tensor \eqref{viscosity}, boundary conditions \eqref{nsk boundary condition} and initial conditions $[\varrho_0, \textbf{J}_0]$ as in the hypotheses of Theorem \eqref{ek existence}. For each fixed $\delta>0$, the existence of a dissipative solution $[\varrho_{\delta}, \textbf{u}_{\delta}]$ in the sense of Definition \ref{nsk dissipative solution} was proven in Theorem \eqref{nsk existence}. Similarly to what was done in the previous section, passing to suitable subsequences as the case may be, we have the following convergences as $\delta \rightarrow 0$.
	\begin{align}
	\varrho_{\delta} \rightarrow \varrho \quad &\mbox{in } C_{\textup{weak}}([0,T]; L^{\gamma}(\Omega)) \\[0.1cm]
	\varrho_{\delta}\textup{\textbf{u}}_{\delta} \rightarrow \textup{\textbf{J}}\quad &\mbox{in } C_{\textup{weak}}([0,T]; L^q(\Omega; \mathbb{R}^d)),  \\[0.1cm]
	\varrho_{\delta} \overset{*}{\rightharpoonup} \varrho \quad &\mbox{in } L^{\infty}(0,T;W^{1, \frac{2\gamma}{\gamma+1}}(\Omega)), \\[0.1cm]
	p(\varrho_{\delta}) \overset{*}{\rightharpoonup} \overline{p(\varrho)} \quad &\mbox{in } L^{\infty}(0,T; \mathcal{M}(\overline{\Omega})),\\[0.1cm]
	\varrho_{\delta}\textup{\textbf{u}}_{\delta} \otimes \textup{\textbf{u}}_{\delta} \overset{*}{\rightharpoonup}  \overline{ \frac{\textbf{J} \otimes \textbf{J}}{\varrho}} \quad &\mbox{in } L^{\infty}(0,T; \mathcal{M}(\overline{\Omega}; \mathbb{R}_{\textup{sym}}^{d\times d})), \\[0.1cm]
	\nabla_x \sqrt{\varrho_{\delta}} \otimes 	\nabla_x \sqrt{\varrho_{\delta}} \overset{*}{\rightharpoonup} \overline{\nabla_x \sqrt{\varrho} \otimes \nabla_x \sqrt{\varrho}} \quad &\mbox{in } L^{\infty}(0,T; \mathcal{M}(\overline{\Omega}; \mathbb{R}_{\textup{sym}}^{d\times d})), \\[0.1cm]
	\mathfrak{R}_{\delta} \overset{*}{\rightharpoonup} \widetilde{\mathfrak{R}} \quad &\mbox{in } L^{\infty}(0,T; \mathcal{M}(\overline{\Omega}; \mathbb{R}_{\textup{sym}}^{d\times d})), \\[0.1cm]
	\sqrt{\delta} \  \mathbb{S}(\nabla_x \textbf{u}_{\delta}) \rightharpoonup \sqrt{\delta} \  \widetilde{\mathbb{S}} \quad &\mbox{in } L^2((0,T)\times \Omega; \mathbb{R}^{d\times d}), \label{convergence viscosities}
	\end{align}
	with $q$ defined as in \eqref{weak convergence momenta}.
	
	We are now ready to let $\delta \rightarrow 0$ in \eqref{nsk weak formulation continuity equation}--\eqref{nsk energy inequality}. Notice that the term with the $\delta$-dependent viscous stress tensor vanishes due to convergence \eqref{convergence viscosities}; indeed,
	\begin{equation*}
	\delta \int_{0}^{\tau} \int_{\Omega} \mathbb{S}(\nabla_x \textbf{u}_{\delta}) : \nabla_x \bm{\varphi} \ \textup{d}x\textup{d}t = \sqrt{\delta} \int_{0}^{\tau} \int_{\Omega} \sqrt{\delta} \  \mathbb{S}(\nabla_x \textbf{u}_{\delta}) : \nabla_x \bm{\varphi} \ \textup{d}x\textup{d}t \rightarrow 0.
	\end{equation*}
	We get the weak formulations of the continuity equation \eqref{ek weak formulation continuity equation}, of the balance of momentum \eqref{ek weak formulation balance of momentum} and of the energy inequality \eqref{ek energy inequality} for the quantum Euler system, with
	\begin{equation*}
	\mathfrak{R} \in L^{\infty}(0,T; \mathcal{M}^+(\overline{\Omega}; \mathbb{R}_{\textup{sym}}^{d\times d}))
	\end{equation*}
	such that
	\begin{align*}
	\textup{d} \mathfrak{R} = \textup{d} \widetilde{\mathfrak{R}}  &+ \left(\overline{p(\varrho)}- p(\varrho)\right) \mathbb{I} \ \textup{d}x + \left( \overline{ \frac{\textbf{J} \otimes \textbf{J}}{\varrho}}- \frac{\textbf{J} \otimes \textbf{J}}{\varrho} \right) \textup{d}x \\
	&+ \hbar \left(\overline{\nabla_x \sqrt{\varrho} \otimes \nabla_x \sqrt{\varrho}} - \nabla_x \sqrt{\varrho} \otimes \nabla_x \sqrt{\varrho}\right) \textup{d}x.
	\end{align*}
	Indeed, proceeding as in the previous section, we can write for any $\bm{\xi} \in \mathbb{R}^d$ and any bounded open set $\mathcal{B} \subset \Omega$
	\begin{equation*}
	\left( \overline{ \frac{\textbf{J} \otimes \textbf{J}}{\varrho}}- \frac{\textbf{J} \otimes \textbf{J}}{\varrho} \right) : (\bm{\xi} \otimes \bm{\xi}) = \overline{ \left| \frac{\textbf{J}\cdot \bm{\xi}}{\varrho} \right|^2} -   \left| \frac{\textbf{J} \cdot \bm{\xi}}{\varrho} \right|^2 \quad \mbox{in } \mathcal{D}'((0,T) \times \mathcal{B}),
	\end{equation*}
	where the non-negativity of the right-hand side quantity will follow from the convexity of the lower semi-continuous function $[\varrho, \textbf{J}] \mapsto \left| \frac{\textbf{J} \cdot \bm{\xi}}{\varrho} \right|^2$. This concludes the proof of Theorem \ref{ek existence}.
	
	\section{Semiflow selection}
	
	We start by fixing a proper setting. We let
	\begin{itemize}
		\item $H:=W^{-k,2}(\Omega) \times W^{-k,2}(\Omega; \mathbb{R}^d) \times \mathbb{R}$ with $k> \frac{d}{2}+1$ fixed; notice that with this particular choice of the constant $k$ we can guarantee
		\begin{equation} \label{compact embedding Sobolev}
			L^p(\Omega) \hookrightarrow\hookrightarrow W^{-k,2}(\Omega) \quad \mbox{for any } p\geq 1;
		\end{equation}
		\item $\mathcal{D}$ denote the space of initial data associated to the quantum Navier--Stokes or quantum Euler systems; in both cases, it can be chosen as
		\begin{equation*}
		\mathcal{D} := \left\{ [\varrho_0, \textbf{J}_0, E_0] \in H: \ \varrho_0 \in L^1(\Omega), \ \varrho_0\geq 0, \ \textbf{J}_0\in L^1(\Omega;\mathbb{R}^d) \mbox{ satisfying } \eqref{initial energy}
		\right\}
		\end{equation*}
		where
		\begin{equation} \label{initial energy}
		\int_{\Omega} \left[ \frac{1}{2} \frac{|\textbf{J}_0|^2}{\varrho_0} + P(\varrho_0) + \frac{\hbar}{2} |\nabla_x \sqrt{\varrho_{0}}|^2\right] {\rm d}x \leq E_0;
		\end{equation}
		\item $\mathcal{T}=\mathfrak{D}([0,\infty); H)$ represents the trajectory space;
		
		\item $\mathcal{U}: \mathcal{D}\rightarrow 2^{\mathcal{T}}$ represents the set--valued mapping that associate to every $[\varrho_0, \textbf{J}_0, E_0] \in \mathcal{D}$ the family of dissipative solutions in the sense of Definition \ref{nsk dissipative solution} or \ref{ek dissipative solution} if we are considering the quantum Navier--Stokes or quantum Euler system, respectively, arising from the initial data $[\varrho_0, \textbf{J}_0, E_0] $. More precisely, for every $[\varrho_0, \textbf{J}_0, E_0] \in \mathcal{D}$
		\begin{align*}
		\mathcal{U}&[\varrho_0, \textbf{J}_0, E_0]= \\
		&\{ [\varrho, \textbf{J}, E] \in \mathcal{T}: \ [\varrho, \textbf{J}, E] \mbox{ is a dissipative solution with initial data } [\varrho_0, \textbf{J}_0, E_0] \}.
		\end{align*}
		Notice that also in the context of the quantum Navier--Stokes system, we consider the momentum $\textbf{J}=\varrho \textbf{u}$ as a state variable along with the density $\varrho$ instead of the velocity $\textbf{u}$ because it is at least weakly continuous in time.
	\end{itemize}
	
	We are now ready to give the following definition.
	\begin{definition}[Semiflow selection] \label{semiflow selection}
		A \textit{semiflow selection} in the class of dissipative solutions is a Borel measurable map $U: \mathcal{D} \rightarrow \mathcal{T}$ such that
		\begin{equation*}
		U[\varrho_0, \textbf{J}_0, E_0] \in \mathcal{U}[\varrho_0, \textbf{J}_0, E_0] \mbox{ for every } [\varrho_0, \textbf{J}_0, E_0] \in \mathcal{D}
		\end{equation*}
		satisfying the semigroup property: for any $[\varrho_0, \textbf{J}_0, E_0] \in \mathcal{D}$ and any $t_1, t_2 \geq 0$
		\begin{equation*}
		U[\varrho_0, \textbf{\textup{J}}_0, E_0] (t_1+t_2)= U[\varrho(t_1), \textbf{\textup{J}}(t_1), E(t_1)] (t_2)
		\end{equation*}
		where $[\varrho, \textbf{J}, E]= U[\varrho_0, \textbf{J}_0, E_0]$.
	\end{definition}
	
	The goal of this section is to prove the following two results.
	
	\begin{theorem}[Semiflow selection for the quantum Navier--Stokes system] \label{nsk semiflow selection}
		The quantum Navier--Stokes system \eqref{nsk continuity equation}--\eqref{nsk balance of momentum} with constitutive relations \eqref{pressure}--\eqref{viscosity} and boundary conditions \eqref{nsk boundary condition} admits a semiflow selection in the sense of Definition \ref{semiflow selection}.
	\end{theorem}
	
	\begin{theorem}[Semiflow selection for the quantum Euler system] \label{ek semiflow selection}
		The quantum Euler system \eqref{ek continuity equation}--\eqref{ek balance of momentum} with the isentropic pressure \eqref{pressure} and boundary conditions \eqref{ek boundary conditions} admits a semiflow selection in the sense of Definition \ref{semiflow selection}.
	\end{theorem}
	
	Both Theorems \ref{nsk semiflow selection} and \eqref{ek semiflow selection} are a direct consequence of Theorem 3.2 in \cite{Bas} once we have verified that the set--valued map $\mathcal{U}$ verifies the following five properties.
	\begin{itemize}
		\item[(\textbf{P1})] \textit{Non-emptiness}: $\mathcal{U}[\varrho_0, \textbf{\textup{J}}_0, E_0] $ is a non-empty subset of $\mathcal{T}$ for any $[\varrho_0, \textbf{\textup{J}}_0, E_0]  \in \mathcal{D}$.
		\item[(\textbf{P2})] \textit{Compactness}: $\mathcal{U}[\varrho_0, \textbf{\textup{J}}_0, E_0] $ is a compact subset of $\mathcal{T}$ for every $[\varrho_0, \textbf{\textup{J}}_0, E_0] \in \mathcal{D}$.
		\item[(\textbf{P3})] \textit{Measurability}: $\mathcal{U}: \mathcal{D}\rightarrow 2^{\mathcal{T}}$ is Borel measurable.
		\item [(\textbf{P4})] \textit{Shift invariance}: introducing the positive shift operator $S_T \circ \Phi$ for every $T>0$ and $\Phi \in \mathcal{T}$ as
		\begin{equation*}
		S_T \circ \Phi(t) = \Phi(T+t), \quad \mbox{for all } t\geq 0,
		\end{equation*}
		then, for any $T>0$, $[\varrho_0, \textbf{\textup{J}}_0, E_0] \in \mathcal{D}$ and $[\varrho, \textbf{\textup{J}}, E] \in \mathcal{U}[\varrho_0, \textbf{\textup{J}}_0, E_0]$, we have
		\begin{equation*}
		S_T \circ [\varrho, \textbf{\textup{J}}, E] \in \mathcal{U} ([\varrho(T), \textbf{\textup{J}}(T), E(T-)]).
		\end{equation*}
		\item[(\textbf{P5})] \textit{Continuation}: introducing the continuation operator $\Phi_1 \cup_T \Phi_2$ for every $T>0$ and $\Phi_1, \Phi_2 \in \mathcal{T}$ as
		\begin{equation*}
		\Phi_1 \cup_T \Phi_2(t) =\begin{cases}
		\Phi_1(t) &\mbox{for } 0\leq t\leq T, \\
		\Phi_2(t-T) &\mbox{for } t>T,
		\end{cases} \quad \mbox{for all }t\geq 0,
		\end{equation*}
		then, for any $T>0$, $[\varrho_0, \textbf{\textup{J}}_0, E_0] \in \mathcal{D}$,
		\begin{align*}
		[\varrho_1, \textbf{\textup{J}}_1, E_1] &\in \mathcal{U}[\varrho_0, \textbf{\textup{J}}_0, E_0], \\
		[\varrho_2, \textbf{\textup{J}}_2, E_2] &\in \mathcal{U}[\varrho_1(T), \textbf{\textup{J}}_1(T), E_1(T-)],
		\end{align*}
		we have
		\begin{equation*}
		[\varrho_1, \textbf{\textup{J}}_1, E_1] \cup_T [\varrho_2, \textbf{\textup{J}}_2, E_2] \in \mathcal{U}[\varrho_0, \textbf{\textup{J}}_0, E_0].
		\end{equation*}
	\end{itemize}
	
	To this end, we have the following facts.
	\begin{itemize}
		\item Property (\textbf{P1}) is equivalent in showing the \textit{existence} of a dissipative solution in the sense of Definitions \ref{nsk dissipative solution} and \ref{ek dissipative solution} for any fixed initial data $[\varrho_0, \textbf{J}_0, E_0] \in \mathcal{D}$. This has already been achieved in Theorems \ref{nsk existence} and \ref{ek existence}.
		\item Properties (\textbf{P2}) and (\textbf{P3}) hold true if we manage to prove the \textit{weak sequential stability} of the solution set $\mathcal{U}[\varrho_0, \textbf{J}_0, E_0]$ for every $[\varrho_0, \textbf{J}_0, E_0] \in \mathcal{D}$ fixed, since it will  in particular imply compactness and the closed-graph property of the mapping
		\begin{equation*}
		\mathcal{D} \ni [\varrho_0, \textbf{J}_0, E_0]  \rightarrow \mathcal{U}[\varrho_0, \textbf{J}_0, E_0] \in 2^{\mathcal{T}},
		\end{equation*}
		and thus the Borel--measurality of $\mathcal{U}$, cf. Lemma 12.1.8 in \cite{StrVar}.
		\item Properties (\textbf{P4}) and (\textbf{P5}) can be easily checked for both systems following the same arguments done in \cite{BreFeiHof}, Lemma 4.2 and 4.3.
	\end{itemize}
	
	Therefore, the proofs of Theorems \ref{nsk semiflow selection} and \ref{ek semiflow selection} reduce to the proof of the weak sequential stability results.
	
	\begin{proposition}[Weak sequential stability for the quantum Navier--Stokes system] \label{nsk weak sequential stability}
		Let \\
		$\{ [\varrho_n, \textbf{\textup{u}}_n] \}_{n \in \mathbb{N}}$ be a  family of dissipative solutions of the quantum Navier--Stokes system \eqref{nsk continuity equation}, \eqref{nsk balance of momentum} with the corresponding total energies  $\{ E_n\}_{n \in \mathbb{N}}$ and initial data $\{ [\varrho_{0,n}, \textbf{\textup{J}}_{0,n}, E_{0,n}] \}_{n \in \mathbb{N}}$ in the sense of Definition \ref{nsk dissipative solution}. If
		\begin{equation*}
		[\varrho_{0,n}, \textbf{\textup{J}}_{0,n}, E_{0,n}] \rightarrow [\varrho_0, \textbf{\textup{J}}_0, E_0] \quad \mbox{in } H,
		\end{equation*}
		then, at least for suitable subsequences,
		\begin{equation} \label{nsk convergence in trajectory space}
		[\varrho_n, \textbf{\textup{J}}_n=\varrho_n\textbf{\textup{u}}_n, E_n] \rightarrow [\varrho, \textbf{\textup{J}}=\varrho\textbf{\textup{u}}, E] \quad \mbox{in } \ \mathfrak{D}([0, \infty); H),
		\end{equation}
		where $[\varrho, \textbf{\textup{u}}]$ is another dissipative solution of the same problem with total energy $E$.
	\end{proposition}
	
	\begin{proposition}[Weak sequential stability for the quantum system] \label{ek weak sequential stability}
		Let $\{ [\varrho_n, \textbf{\textup{J}}_n] \}_{n \in \mathbb{N}}$ be a  family of dissipative solutions of the quantum Euler system \eqref{ek continuity equation}, \eqref{ek balance of momentum} with the corresponding total energies  $\{ E_n\}_{n \in \mathbb{N}}$ and initial data $\{ [\varrho_{0,n}, \textbf{\textup{J}}_{0,n}, E_{0,n}] \}_{n \in \mathbb{N}}$ in the sense of Definition \ref{ek dissipative solution}. If
		\begin{equation*}
		[\varrho_{0,n}, \textbf{\textup{J}}_{0,n}, E_{0,n}] \rightarrow [\varrho_0, \textbf{\textup{J}}_0, E_0] \quad \mbox{in } H,
		\end{equation*}
		then, at least for suitable subsequences,
		\begin{equation} \label{ek convergence in trajectory space}
		[\varrho_n, \textbf{\textup{J}}_n, E_n] \rightarrow [\varrho, \textbf{\textup{J}}, E] \quad \mbox{in } \ \mathfrak{D}([0, \infty); H),
		\end{equation}
		where $[\varrho, \textbf{\textup{J}}]$ is another dissipative solution of the same problem with total energy $E$.
	\end{proposition}
	
	We are not going to show the two aforementioned propositions in details since the proofs would be essentially a repetition of what was done in Sections \ref{Limit n to infinity} and \ref{Existence Euler}. We just point out that the convergences
	\begin{equation*}
		[\varrho_n, \textbf{J}_n] \rightarrow [\varrho, \textbf{J}] \quad \mbox{in }  C_{\textup{weak,loc}}([0,\infty); L^p(\Omega) \times L^q(\Omega; \mathbb{R}^d))
	\end{equation*}
	can be strengthened to
	\begin{equation*}
		[\varrho_n, \textbf{J}_n] \rightarrow [\varrho, \textbf{J}] \quad \mbox{in }  C_{\textup{loc}}([0,\infty); W^{-k,2}(\Omega) \times W^{-k,2}(\Omega; \mathbb{R}^d))
	\end{equation*}
	thanks to the compact embedding \eqref{compact embedding Sobolev}, implying in particular that,
	\begin{equation*}
		[\varrho_n, \textbf{J}_n] \rightarrow [\varrho, \textbf{J}] \quad \mbox{in }  \mathfrak{D}([0,\infty); W^{-k,2}(\Omega) \times W^{-k,2}(\Omega; \mathbb{R}^d)),
	\end{equation*}
	as for continuous functions, the convergence in the Skorokhod space coincides with the uniform one, cf. condition (ii) of Proposition 2.1 in \cite{Bas}. Moreover, the energies $\{ E_n \}_{n\in \mathbb{N}}$ are non-increasing functions, locally of bounded variation; therefore, from Helly's selection theorem, there exists a subsequence converging pointwise
	\begin{equation*}
		E_n(t) \rightarrow E(t) \quad \mbox{for all } t \in [0,\infty),
	\end{equation*}
	implying in particular that
	\begin{equation*}
		E_n \rightarrow E \quad \mbox{in } \mathfrak{D}([0,\infty)),
	\end{equation*}
	as for monotone functions, the convergence in the Skorokhod space coincides with the almost everywhere one, cf. condition (i) of Proposition 2.1 in \cite{Bas}.
	\appendix
	
	\section{Appendix}
	
	\subsection{Function spaces} \label{Function spaces}
	
	Let $Q\subseteq \mathbb{R}^N$, $N\geq 1$, be an open set, $X$ a Banach space and $M\geq 1$. We denote with
	\begin{itemize}
		\item $C_{\textup{weak}}(Q; X)$ the space of functions defined on $Q$ and ranging in $X$ which are continuous with respect to the weak topology. If $Q$ is bounded, we say that $f_n \rightarrow f \ \mbox{in } C_{\textup{weak}}(\overline{Q}; X)$ as $n\rightarrow \infty$ if for all $g\in X^*$
		\begin{equation*}
		\sup_{y\in \overline{Q}} \left|\langle g; f_n(y)-f(y) \rangle_{X^*, X} \right| \rightarrow 0 \quad \mbox{as } n\rightarrow \infty;
		\end{equation*}
		 \item $C^k(Q; X)$, with $k$ a non-negative integer, the space of $k$-times continuously differentiable functions on $Q$ and $C^{\infty}(Q; X) = \bigcap_{k=0}^{\infty} C^k(Q; X)$;
		 \item $\mathcal{D}(Q; X)=C^{\infty}_c(Q; X)$ the space of functions belonging to $C^{\infty}(Q; X)$ and having compact support in $Q$;
		 \item $\mathcal{D}'(Q; \mathbb{R}^M)= [C^{\infty}_c(Q; \mathbb{R}^M)]^*$ the space of distributions;
		 \item $\mathcal{M}(Q; \mathbb{R}^M)= \big[\overline{C_c(Q; \mathbb{R}^M)}^{\| \cdot \|_{\infty}} \big]^*$ the space of vector-valued Radon measures. If $\Omega \subset \mathbb{R}^N$ is a bounded domain, then $\mathcal{M}(\overline{\Omega})= [C(\overline{\Omega})]^*$.
		
		 \item $\mathcal{M}^+(Q)$ the space of positive Radon measures;
		
		 \item $\mathcal{M}^+(Q; \mathbb{R}^{N\times N}_{\textup{sym}})$ the space of tensor--valued Radon measures $\mathfrak{R}$ such that $\mathfrak{R}: (\xi \otimes \xi) \in \mathcal{M}^+(Q)$ for all $\xi \in \mathbb{R}^d$, and with components $\mathfrak{R}_{i,j}=\mathfrak{R}_{j,i}$;
		 \item $L^p(Q;X)$, with $1\leq p\leq \infty$, the Lebesgue space defined on $Q$ and ranging in $X$;
		 \item $W^{k,p}(Q; \mathbb{R}^M)$, with $1\leq p\leq \infty$ and $k$ a positive integer, the Sobolev space defined on $Q$;
		 \item $W^{-k, p'}(Q; \mathbb{R}^m)$, with $p'$ the conjugate exponent of $1\leq p< \infty$ and $k$ a positive integer, the dual space of $W_0^{k,p}(Q; \mathbb{R}^m)=\big[\overline{C_c(Q; \mathbb{R}^M)}^{\| \cdot\|_{W^{k,p}(Q; \mathbb{R}^M)}} \big]^*$;
		 \item $\mathfrak{D}([0,\infty); H)$ the Skorokhod space of c\`{a}gl\`{a}d functions defined on $[0,\infty)$ taking values in a Hilbert space $H$. More precisely, $\Phi$ belongs to the space $\mathfrak{D}([0,\infty); H)$ if it is left--continuous and has right--hand limits:
		 \begin{itemize}
		 	\item[(i)] for $t>0$, \ $\Phi(t-)=\lim_{s \uparrow t} \Phi(s)$ exists \ and \  $\Phi(t-)=\Phi(t)$;
		 	\item[(ii)] for $t\geq 0$, \  $\Phi(t+)=\lim_{s\downarrow t} \Phi(s)$ exists.
		 \end{itemize}
	\end{itemize}
	
	\subsection{Energy} \label{Energies}
	
	In this section, we will show how to deduce the total energy balances \eqref{nsk total energy balance} and \eqref{ek total energy balance}. First of all, introducing the \textit{drift velocity} $\textbf{v}=\textbf{v}(\varrho, \nabla_x \varrho)$ such that
	\begin{equation} \label{vector q}
	\textbf{v}= \frac{\nabla_x \sqrt{\varrho}}{\sqrt{\varrho}},
	\end{equation}
	and taking the gradient in the continuity equations \eqref{nsk continuity equation}, \eqref{ek continuity equation}, we get extra equations for $\textbf{v}$:
	\begin{equation} \label{nsk derived continuity equation}
	\partial_t(\varrho \textbf{v}) + \divv_x(\varrho \textbf{v} \otimes \textbf{u}) + \frac{1}{2} \divv_x \big( \varrho \nabla_x^{\top} \textbf{u} \big)=0,
	\end{equation}
	when considering system \eqref{nsk continuity equation}--\eqref{nsk balance of momentum}, and
	\begin{equation}
	\partial_t(\varrho \textbf{v}) + \frac{1}{2}\divv_x \nabla_x^{\top} \textbf{J}=0,
	\end{equation}
	when considering system \eqref{ek continuity equation}--\eqref{ek balance of momentum}. Furthermore, notice that we can write
	\begin{equation*}
	\mathbb{K}(\varrho, \nabla_x\textbf{v})= \frac{\hbar}{2} \varrho \nabla_x \textbf{v}.
	\end{equation*}
	
	Supposing that all the quantities in question are smooth, we can multiply the balance of momentum \eqref{nsk balance of momentum} of the quantum Navier--Stokes system by $\textbf{u}$ and, using the continuity equation \eqref{nsk continuity equation}, we can deduce
	\begin{equation} \label{expression 2}
	\begin{aligned}
	\partial_t \left( \frac{1}{2}\varrho |\textbf{u}|^2 \right) + \divv_x \left( \left[ \frac{1}{2}\varrho |\textbf{u}|^2 +p(\varrho)\right] \textbf{u} \right) - p(\varrho) \divv_x \textbf{u} &+ \mathbb{S}(\nabla_x \textbf{u}): \nabla_x \textbf{u} + \frac{\hbar}{2} \varrho \nabla_x \textbf{v}: \nabla_x \textbf{u} \\
	&= \divv_x \left(\mathbb{S}(\nabla_x \textbf{u}) \cdot \textbf{u}  + \mathbb{K}(\varrho, \nabla_x \textbf{v}) \cdot \textbf{u} \right).
	\end{aligned}
	\end{equation}
	Similarly, we multiply \eqref{nsk derived continuity equation} by $\textbf{v}$ to get
	\begin{equation} \label{expression 1}
	\partial_t \left( \frac{1}{2}\varrho |\textbf{v}|^2 \right) + \divv_x \left( \frac{1}{2}\varrho |\textbf{v}|^2 \textbf{u} \right) - \frac{1}{2}\varrho \nabla_x \textbf{u}: \nabla_x \textbf{v}= - \frac{1}{2}\divv_x \big( \varrho \nabla_x^{\top} \textbf{u} \cdot \textbf{v}\big),
	\end{equation}
	where we used the fact that
	\begin{equation} \label{identity}
	\nabla_x^{\top} \textbf{u}: \nabla_x \textbf{v}= \nabla_x \textbf{u}: \nabla_x^{\top} \textbf{v} = \nabla_x \textbf{u}: \nabla_x \textbf{v}
	\end{equation}
	since $\nabla_x \textbf{v}$ is symmetric. Multiplying \eqref{expression 1} by $\hbar$ and summing the obtained identity to \eqref{expression 2} we get
	\begin{align*}
	\partial_t \left( \frac{1}{2}\varrho |\textbf{u}|^2 + \frac{\hbar}{2} \varrho |\textbf{v}|^2\right) &+ \divv_x \left( \left[ \frac{1}{2}\varrho |\textbf{u}|^2 +p(\varrho) + \frac{\hbar}{2} |\textbf{v}|^2\right] \textbf{u} \right) - p(\varrho) \divv_x \textbf{u} + \mathbb{S}(\nabla_x \textbf{u}): \nabla_x \textbf{u} \\
	&\hspace{3.5cm}= \divv_x \left(\mathbb{S}(\nabla_x \textbf{u}) \cdot \textbf{u}  + \mathbb{K}(\varrho, \nabla_x\textbf{v}) \cdot \textbf{u} - \frac{\hbar}{2} \varrho \nabla_x^{\top} \textbf{u} \cdot \textbf{v} \right).
	\end{align*}
	Recalling that the pressure potential $P=P(\varrho)$ is characterized by \eqref{equation pressure potential}, from the continuity equation \eqref{nsk continuity equation} we obtain the following identity
	\begin{equation*}
	-p(\varrho) \divv_x \textbf{u} = \partial_t P(\varrho) +\divv_x [P(\varrho) \textbf{u}].
	\end{equation*}
	Now, it is enough to integrate over $\Omega$ and use the boundary conditions to get the desired expressions.


\begin{thebibliography}{}
		
		\bibitem{AbbFeiNov}
		A. Abbatiello, E. Feireisl and A. Novotn\'{y},
		\textit{Generalized solutions to mathematical models of compressible viscous fluids}, Discrete \& Continuous Dynamical Systems \textbf{41}(1): 1--28; 2021
		
		\bibitem{AntMar}
		P. Antonelli and P. Marcati,
		\textit{On the finite energy weak solutions to a system in quantum fluid dynamics},
		Commun. Math. Phys. \textbf{287}: 657--686; 2009


        \bibitem{AntSpi}
		P. Antonelli and S. Spirito,
		\textit{Global existence of finite energy weak solutions of quantum Navier--Stokes equations}, Arch. Rational Mech. Anal. \textbf{225}: 1161--1199; 2017
		
		
       \bibitem{a3}

      	P. Antonelli and S. Spirito,
		\textit{On the compactness of weak solutions to the Navier-Stokes-Korteweg equations for capillary fluids}, Nonlinear Anal. \textbf{187}: 110--124; 2019.

        \bibitem{a4}
        P. Antonelli, C. G. Cianfarani, C. Lattanzio and S. Spirito,
        \textit{Relaxation limit from the quantum Navier-Stokes equations to the quantum drift-diffusion equation}, J. Nonlinear Sci., \textbf{31}(71); 2021		

	    \bibitem{AudHas}
		C. Audiard and B. Haspot,
		\textit{Global well-posedness of the Euler Korteweg system for small irrotational data}, Commun. Math. Phys. \textbf{315}: 201--247; 2017
		
		\bibitem{Bas}
		D. Basari\'{c},
		\textit{Semiflow selection to models of general compressible viscous fluids}, J. Math. Fluid Mech. \textbf{23}(2); 2021
		
		\bibitem{BreFeiHof1}
		D. Breit, E. Feireisl  and M. Hofmanov\'{a},
		\textit{Markov selection for the stochastic compressible Navier--Stokes system}, Ann. Appl. Probab. \textbf{30}(6): 2547--2572; 2020
		
		\bibitem{BreFeiHof}
		D. Breit, E. Feireisl and M. Hofmanov\'{a},
		\textit{Solution semiflow to the isentropic Euler system}, Arch. Rational Mech. Anal. \textbf{235}: 167--194; 2020
		
		\bibitem{BreGisLac}
		D. Bresch, M. Gisclon and I. Lacroix-Violet,
		\textit{On the Navier--Stokes--Korteweg and Euler--Korteweg system}, Arch. Rational Mech. Anal. \textbf{223}: 975--1025; 2019
		
		\bibitem{BruMeh}
		S. Brull and F. M\'{e}hats,
		\textit{Derivation of viscous correction terms for the isothermal quantum Euler model}, ZAMM Z. Angew. Math. Mech. \textbf{90}: 219--230; 2010
		
		\bibitem{ca}	
        M. Caggio and D. Donatelli,
        \textit{High Mach number limit for Korteweg fluids with density dependent viscosity}, J. Differential Equations  \textbf{277}: 1--37; 2021.


        \bibitem{CarKap}
		J. E. Cardona and L. Kapitanski,
		\textit{Semiflow selection and Markov selection theorems}, Topol. Methods Nonlinear Anal. \textbf{56}(1): 197--227; 2020
	     
         
		\bibitem{ChaJinNov}
		T. Chang, B. J. Jin and A. Novotn\'{y}, \textit{Compressible Navier-Stokes system with inflow-outflow boundary data}, SIAM J. Math. Anal. \textbf{51}(2): 1238--1278; 2019
		

         \bibitem{c}
         C. G. Cianfarani and C. Lattanzio,
         \textit{High friction limit for Euler-Korteweg and Navier-Stokes-Korteweg models via relative entropy approach}, J. Differential Equations, \textbf{269}: 10495-10526; 2020.


		\bibitem{DonFeiMar}
		D. Donatelli, E. Feireisl and P. Marcati,
		\textit{Well/ill posedness for the Euler--Korteweg--Poisson system and related problems}, Commun. Partial Differ. Equ. \textbf{40}: 1314--1335; 2015
		
        \bibitem{d1} 
        D. Donatelli and P. Marcati,
        \textit{Quasi-neutral limit, dispersion, and oscillations for Korteweg-type fluids}, SIAM J. Math. Anal., \textbf{47}:2265--2282; 2015.

        \bibitem{d2}
        D. Donatelli and P. Marcati,
        \textit{Low Mach number limit for the quantum hydrodynamics system}, Res. Math. Sci., \textbf{3}(13); 2016


		\bibitem{Don}
		J. Dong,
		\textit{A note on barotropic compressible quantum Navier--Stokes equations}, Nonlinear Anal. Real World Appl. \textbf{73}: 854--856; 2010
		
		\bibitem{Fei1}
		E. Feireisl, \textit{Dynamics of viscous compressible fluids}, Oxford University Press, Oxford; 2003
		
		\bibitem{Fei}
		E. Feireisl,
		\textit{On weak--strong uniqueness for the compressible Navier--Stokes system with non-monotone pressure law}, Commun. Partial Differ. Equ. \textbf{44}(3): 271--278; 2019
		
		\bibitem{FeiGwiSwiWie}
		E. Feireisl, P. Gwiazda, A. \'{S}wierczewska-Gwiazda and E. Wiedemann,
		\textit{Dissipative measure-valued solutions to the compressible Navier--Stokes system}, Calc. Var. Partial Differ. Equ. \textbf{55} (6): 55--141; 2016
		
		\bibitem{FeiJinNov}
		E. Feireisl, B. J. Jin and A. Novotn\'{y},
		\textit{Relative Entropies, Suitable Weak Solutions, and Weak--Strong Uniqueness for the Compressible Navier--Stokes System}, Journal of Mathematical Fluid Dynamics \textbf{14}: 717--730; 2012
		
		\bibitem{FeiLuk}
		E. Feireisl and M. Luk\'{a}\v{c}ov\'{a}-Medvi\softd ov\'{a},
		\textit{Convergence of a mixed finite element--finite volume scheme for the isentropic Navier--Stokes system via dissipative measure--valued solutions}, Found. Comput. Math. \textbf{18}: 703--730; 2018
		
		\bibitem{FeiNov}
		E. Feireisl and A. Novotn\'{y},
		\textit{Weak--strong uniqueness property for models of compressible viscous fluids near vacuum}, Nonlinearity \textbf{34}(9); 2021
		
		\bibitem{FerZho}
		D. K. Ferry and J.-R. Zhou,
		\textit{Form of the quantum potential for use in hydrodynamic equations for semiconductor device modeling}, Phys. Rev. B \textbf{48}: 7944--7950; 1993
		
		\bibitem{FlaRom}
		F. Flandoli and M. Romito,
		\textit{Markov selections for the 3D stochastics Navier--Stokes equations}, Probab. Theory Related Fields, \textbf{140}(3-4): 407--458; 2008
		
		\bibitem{Ger}
		P. Germain,
		\textit{Weak--strong uniqueness for the isentropic compressible Navier--Stokes system}, J. Math. Fluid Mech. \textbf{13}(1): 137--146; 2011
		
		\bibitem{GieLatTza}
		J. Giesselmann, C. Lattanzio and A. E. Tzavaras,
		\textit{Relative energy for the Korteweg theory and related Hamiltonian flows in gas dynamics}, Arch. Rational Mech. Anal. \textbf{223}: 1427--1484; 2017
		
		\bibitem{GisLac}
		M. Gisclon and I. Lacroix-Violet,
		\textit{About the barotropic compressible quantum Navier--Stokes equations},
		Nonlinear Anal. Theory Methods Appl. \textbf{128}: 106--121; 2015
		
		\bibitem{GwiSwiWie}
		P. Gwiazda, A. \'{S}wierczewska-Gwiazda and E. Wiedemann,
		\textit{Weak-strong uniqueness for measure-valued solutions of some compressible fluid models}, Nonlinearity \textbf{28}: 3873--3890; 2015
		
		\bibitem{Jia}
		F. Jiang,
		\textit{A remark on weak solutions to the barotropic compressible quantum Navier--Stokes equations}, Nonlinear Anal. Real World Appl. \textbf{12}: 1733--1735; 2011
		
       \bibitem{Ju}
        A. J\"{u}ngel and H. L. Li,
        \textit{Quantum Euler-Poisson systems: global existence and exponential decay}, Quart. Appl. Math., \textbf{62}: 569--600; 2004.


		\bibitem{Jun}
		A. J\"{u}ngel,
		\textit{Global weak solutions to compressible Navier- Stokes equations for quantum fluids}, SIAM J. Math. Anal. \textbf{42}: 1025--1045; 2010
		

		\bibitem{Kry}
		N. V. Krylov,
		\textit{The selection of a Markov process from a Markov system of processes, and the construction of quasidiffusion processes}, Izv. Akad. Nauk SSSR Ser. Mat., \textbf{37}: 691--708; 1973
		
		\bibitem{LacVas}
		I. Lacroix-Violet and A. Vasseur,
		\textit{Global weak solutions to the compressible quantum Navier--Stokes equations and its semi-classical limit}, J. Math. Pures Appl. \textbf{114}(9): 191--210; 2018
		
        \bibitem{li}
        H. L. Li and P. Marcati,
        \textit{Existence and asymptotic behavior of multi-dimensional quantum hydrodynamic model for semiconductors}, Comm. Math. Phys., \textbf{245}: 215--247; 2004.

        \bibitem{l}
        J. Li and Z.P. Xin,
        \textit{Global existence of weak solutions to the barotropic compressible Navier-Stokes flows with degenerate viscosities}, arXiv:1504.06826; 2015


		\bibitem{LofMor}
		M. I. Loffredo and L. M. Morato,
		\textit{On the creation of quantized vortex lines in rotating He--2}, Nuovo Cim. B \textbf{108}: 205--216; 1993
		
		\bibitem{Pro}
		G. Prodi,
		\textit{Un teorema di unicit\'{a} per le equazioni di Navier--Stokes}, Ann. Mat. Pura Appl. \textbf{48}(4): 173--182; 1959
		
		\bibitem{SlaTse}
		R. Slavchov and R. Tsekov,
		\textit{Quantum hydrodynamics of electron gases}, J. Chem. Phys. \textbf{132}; 2010
		
		\bibitem{StrVar}
		D. W. Stroock and S. R. S. Varadhan,
		\textit{Multidimensional diffusion processes}, Classics in Mathematics, Springer--Verlag, Berlin; 2006
		
		\bibitem{VasYu}
		A. Vasseur and C. Yu,
		\textit{Global weak solutions to compressible Navier--Stokes equations with damping}, SIAM J. Math. Anal. \textbf{48}(2): 1489--1511; 2016
		
		\bibitem{Wie}
		E. Wiedemann,
		\textit{Weak--strong uniqueness in fluid dynamics}, Partial Differential Equations in Fluid Dynamics: 289--326; 2018
		
		\bibitem{Wya}
		R. Wyatt,
		\textit{Quantum dynamics with trajectories: introduction to quantum hydrodynamics}, Springer, New York; 2005
		
		\bibitem{ZarTso}
		E. Zaremba and H. C. Tso,
		\textit{Thomas--Fermi--Dirac–von Weizs\"{a}cker hydrodynamics in parabolic wells}, Phys. Rev. B \textbf{49}; 1994
		
	\end{thebibliography}
\end{document}